\documentclass[a4paper,12pt]{amsart}

\usepackage{amsmath,amssymb}
\usepackage{ifthen}
\usepackage{graphicx}
\usepackage[T1]{fontenc}
\usepackage[utf8]{inputenc}
\usepackage[usenames,dvipsnames]{color}
\usepackage[english]{babel}
\usepackage{fancyhdr}
\usepackage{fancybox}
\usepackage{tikz}
\usepackage{cite}
\usepackage{fontawesome5}
\usepackage{marvosym} 

\usepackage[backref=page, colorlinks=true, linkcolor=blue, citecolor=red, urlcolor=blue]{hyperref}

\renewcommand*{\backref}[1]{}
\renewcommand*{\backrefalt}[4]{%
	\ifcase #1 %
	(Not cited.)%
	\or
	[p.#2].%
	\else
	[p.#2].%
	\fi
}

\nonstopmode 
\numberwithin{equation}{section}
\theoremstyle{plain}
\newtheorem{prop}{Proposition}

\newtheorem{conj}{Conjecture}

\theoremstyle{definition}

\newtheorem{exm}{Example}[section]

\newtheorem{cor}{Corollary}[section]
\newtheorem{thm}{Theorem}[section]

\newtheorem{lem}{Lemma}[section]
\newtheorem{prob}{Problem}
\newtheorem{rem}{Remark}[section]

\theoremstyle{plain}
\newtheorem*{theoA}{Theorem A}

\newtheorem*{theoC}{Theorem C}
\newtheorem*{theoD}{Theorem D}
\newtheorem*{theoE}{Theorem E}
\newtheorem*{theoF}{Theorem F}
\newtheorem*{theoG}{Theorem G}

\newtheorem*{lemA}{Lemma A}
\newtheorem*{lemB}{Lemma B}

\newcounter{minutes}
\divide\time by 60
\newcounter{hours}
\multiply\time by 60
\addtocounter{minutes}{-\time}

\newcounter {own}
\def\theown {\thesection.\arabic{own}}

\newenvironment{pf}[1][]{%
	\vskip 3mm
	\noindent
	\ifthenelse{\equal{#1}{}}%
	{{\slshape Proof. }}%
	{{\slshape #1.} }%
}%
{\qed\bigskip}

\newcounter{alphabet}

\newcommand{\B}{{\mathcal B}}

\newcommand{\D}{{\mathbb D}}

\def\be{\begin{equation}}
	\def\ee{\end{equation}}

\newcommand{\bee}{\begin{enumerate}}
	\newcommand{\eee}{\end{enumerate}}

\newcommand{\blem}{\begin{lem}}
	\newcommand{\elem}{\end{lem}}
\newcommand{\bthm}{\begin{thm}}
	\newcommand{\ethm}{\end{thm}}
\newcommand{\bcor}{\begin{cor}}
	\newcommand{\ecor}{\end{cor}}
\newcommand{\beg}{\begin{examp}}
	\newcommand{\eeg}{\end{examp}}
\newcommand{\begs}{\begin{examples}}
	\newcommand{\eegs}{\end{examples}}
\newcommand{\bdefn}{\begin{defn}}
	\newcommand{\edefn}{\edefn}
	\newcommand{\bprob}{\begin{prob}}
		\newcommand{\eprob}{\end{prob}}
	\newcommand{\bei}{\begin{itemize}}
		\newcommand{\eei}{\end{itemize}}
	\newcommand{\bcon}{\begin{conj}}
		\newcommand{\econ}{\end{conj}}
	\newcommand{\bprop}{\begin{prop}}
		\newcommand{\eprop}{\eprop}
		\newcommand{\br}{\begin{rem}}
			\newcommand{\er}{\end{rem}}
		\newcommand{\bpf}{\begin{pf}}
			\newcommand{\epf}{\end{pf}}
		\newcommand{\ba}{\begin{array}}
			\newcommand{\ea}{\end{array}}
		\newcommand{\beq}{\begin{eqnarray}}
			\newcommand{\beqq}{\begin{eqnarray*}}
				\newcommand{\eeq}{\end{eqnarray}}
			\newcommand{\eeqq}{\end{eqnarray*}}

\begin{document}

\title{Hölder Continuity and Boundedness of Composition Operators on Pluriharmonic $\alpha$-Bloch Spaces}

\author{Molla Basir Ahamed}
\address{Molla Basir Ahamed, 
	Department of Mathematics, 
	Jadavpur University, 
	Kolkata-700032, 
	West Bengal, India.}
\email{mbahamed.math@jadavpuruniversity.in}

\author{Vasudevarao Allu$ ^* $}
\address{Vasudevarao Allu, 
	School of Basic Science, 
	Indian Institute of Technology Bhubaneswar, 
	Kolkata-752050, 
	Odisha, India.}
\email{avrao@iitbbs.ac.in}

\subjclass[{AMS} Subject Classification:]{Primary 32A18, 47B33; Secondary 30H30, 32A37.}

\def\thefootnote{}
\footnotetext{ {\tiny File:~\jobname.tex,
printed: \number\year-\number\month-\number\day,
          \thehours.\ifnum\theminutes<10{0}\fi\theminutes }
} \makeatletter\def\thefootnote{\@arabic\c@footnote}\makeatother
\begin{abstract} 
In this paper, we study the metric properties of pluriharmonic $\alpha$-Bloch-type mappings $f=h+\overline{g}\in\mathcal{H}(\mathbb{B}^{n},\mathbb{C}^{n})$ on the unit ball $\mathbb{B}^{n}\subset\mathbb{C}^{n}$. By synthesizing local real-part estimates with the subharmonicity of $\log(\vert{} \det Dh(z) \vert{} + \vert{} \det Dg(z) \vert{})$ along hyperbolic geodesics, we establish a globalized non-trivial lower bound across the entire unit ball $\mathbb{B}^{n}$. Using this structural decay estimate, we prove that the invariant weight quantity $D_{f}^{n,\alpha}(z)$ satisfies a global H\"older continuity estimate of order $1/n$ with respect to the Bergman metric $\beta$. Furthermore, we investigate composition operators $C_{\varphi}(f)=f\circ\varphi$ induced by holomorphic self-mappings $\varphi\in\mathcal{A}(\mathbb{B}^{n},\mathbb{B}^{n})$. We establish both sharp sufficient and necessary sampling conditions for $C_{\varphi}$ to be bounded from below on the pluriharmonic $\alpha$-Bloch space $\mathcal{B}_{P(n,\alpha)}$.
\end{abstract}
\maketitle
\pagestyle{myheadings}
\markboth{M. B. Ahamed and V. Allu}{Hölder Continuity and Boundedness of Composition Operators}
\section{\bf Introduction}
Let $\mathbb{C}^n$ denote the complex Euclidean space of $n \ge 1$ complex variables. For $z = (z_1, z_2, \dots, z_n) \in \mathbb{C}^n$ and $w = (w_1, w_2, \dots, w_n) \in \mathbb{C}^n$, the inner product is defined by
\[
\langle z, w \rangle = \sum_{j=1}^{n} z_j \overline{w}_j
\]
and the norm is given by $|z| := \langle z, z \rangle^{1/2}$. For $x \in \mathbb{C}^n$ and $r > 0$, let us write $\mathbb{B}^n(x, r) := \{z \in \mathbb{C}^n : |z - x| < r\}$ and $S^{n-1}(x, r) := \partial \mathbb{B}^n(x, r)$. Let $\mathbb{B}^n(r) := \mathbb{B}^n(0, r)$, $\partial \mathbb{B}^n(r) := \partial \mathbb{B}^n(0, r)$, $\mathbb{B}^n := \mathbb{B}^n(1)$, and denote by $\partial \mathbb{B}^n$ the boundary of the unit ball. In the case of one complex variable, $\mathbb{D} := \mathbb{B}^1$ is the usual unit disk.\vspace{1.2mm}

The class of all holomorphic functions from $\mathbb{B}^n$ into $\mathbb{C}^n$ is denoted by $\mathcal{A}(\mathbb{B}^n, \mathbb{C}^n)$. We denote by $\text{Aut}(\mathbb{B}^n)$ the automorphism group consisting of all biholomorphic self-mappings of the unit ball $\mathbb{B}^n$.

\subsection{Pluriharmonic mappings}
A twice continuously differentiable function $f$ on $\mathbb{B}^n$ is said to be pluriharmonic if its restriction to every complex line is harmonic. This is equivalent to the condition that, for any $z \in \mathbb{B}^n$,
\begin{align*}
	\frac{\partial^2 f}{\partial z_i \partial \overline{z}_j}(z) = 0, \quad \forall i, j = 1, 2, \dots, n.
\end{align*}
Denote by $\mathcal{H}(\mathbb{B}^n, \mathbb{C}^n)$ the class of all pluriharmonic mappings from $\mathbb{B}^n$ into $\mathbb{C}^n$. Every pluriharmonic mapping $f \in \mathcal{H}(\mathbb{B}^n, \mathbb{C}^n)$ admits a representation $f = h + \overline{g}$, where $h, g \in \mathcal{A}(\mathbb{B}^n, \mathbb{C}^n)$, and the representation is unique if $g(0) = 0$.\vspace{1.2mm}

\noindent For a pluriharmonic mapping $f$ in $\mathbb{B}^n$, let
\begin{align*}
	J_f(z) := \begin{pmatrix} Dh(z) & \overline{Dg(z)} \\ Dg(z) & \overline{Dh(z)} \end{pmatrix}
\end{align*}
be the Jacobian matrix of $f$, where $Dh(z)$ and $Dg(z)$ denote the complex Jacobian matrices of $h$ and $g$ at $z$. Let
\begin{align*}
	\Omega_f := \{z \in \mathbb{B}^n : \det Dh(z) \neq 0\}.
\end{align*}
\subsection{Bergman metric}
For $z \in \mathbb{B}^n$, let
\begin{align*}
	B(z) := \frac{(1 - |z|^2) I + A(z)}{(1 - |z|^2)^2}
\end{align*}
be the Bergman matrix, where $I$ is the $n \times n$ identity matrix, and
\begin{align*}
	A(z) := \begin{pmatrix}
		z_1 \overline{z}_1 & \cdots & z_1 \overline{z}_n \\
		\vdots & \ddots & \vdots \\
		z_n \overline{z}_1 & \cdots & z_n \overline{z}_n
	\end{pmatrix}.
\end{align*}
For a smooth curve $\gamma : [0, 1] \to \mathbb{B}^n$, let
\begin{align*}
	\mathcal{L}_\beta(\gamma) := \int_0^1 \langle B(\gamma(t))\gamma'(t), \gamma'(t) \rangle^{1/2} \, dt.
\end{align*}
For $z, w \in \mathbb{B}^n$, define
\begin{align*}
	\beta(z, w) := \inf_{\gamma} \mathcal{L}_\beta(\gamma),
\end{align*}
where the infimum is taken over all piecewise smooth rectifiable curves $\gamma$ connecting $z$ and $w$. We call $\beta$ the Bergman metric, and note that this metric can be represented as (see \cite[Page 25]{Zhu-Book-2005})
\begin{align*}
	\beta(z, w) = \frac{1}{2} \log \frac{1 + |\phi_z(w)|}{1 - |\phi_z(w)|}, \quad z, w \in \mathbb{B}^n,
\end{align*}
where $\phi_z(w)$ is the involutive automorphism that interchanges the points $0$ and $z$. In particular, for any $w \in \mathbb{B}^n$, the automorphism $\phi_w : \mathbb{B}^n \to \mathbb{B}^n$ that interchanges $0$ and $w$ is given by
\begin{equation}\label{eq:1.1}
	\phi_w(z) = \frac{w - P_w z - \sqrt{1 - |w|^2} Q_w z}{1 - \langle z, w \rangle}, \quad z \in \mathbb{B}^n,
\end{equation}
where $P_w$ is the orthogonal projection of $\mathbb{C}^n$ onto the subspace spanned by $w$, and $Q_w := I - P_w$ is the projection onto the orthogonal complement. It has the following useful property,
\begin{equation}\label{eq:1.2}
	1 - |\phi_w(z)|^2 = \frac{(1 - |z|^2)(1 - |w|^2)}{|1 - \langle z, w \rangle|^2}, \quad z \in \mathbb{B}^n.
\end{equation}
For any $z, w \in \mathbb{B}^n$, the pseudo-hyperbolic distance between $z$ and $w$ is given by $\rho(z, w) := |\phi_z(w)|$. It is well-known that the pseudo-hyperbolic distance is invariant under automorphisms $ \mathbb{B}^n $ \emph{i.e.,}
\begin{align*}
	\rho(g(z), g(w)) = \rho(z, w)\;g \in \text{Aut}(\mathbb{B}^n).
\end{align*} By the definition of the Bergman metric, we obtain, for all $n \ge 1$,
\begin{align*}
	\tanh \beta(z, w) = \rho(z, w), \quad \forall z, w \in \mathbb{B}^n.
\end{align*}
In particular, if $n = 1$ (\textit{i.e.,} if $z, w \in \mathbb{D}$), the automorphism has the explicit form
\begin{equation}\label{eq:1.3}
	\phi_w(z) = \frac{w - z}{1 - \overline{w}z},
\end{equation}
and the pseudo-hyperbolic distance reduces to the classical form
\begin{align*}
	\rho(z, w) = \left| \frac{z - w}{1 - \overline{z}w} \right|.
\end{align*}
In addition, a straightforward computation gives the useful identity
\begin{align*}
	1 - \rho(z,w)^2 = \frac{(1 - |z|^2)(1 - |w|^2)}{|1 - \overline{z}w|^2} = (1 - |w|^2)|\phi_z'(w)|.
\end{align*}
\subsection{Bloch-type spaces on pluriharmonic mappings}
For $f \in \mathcal{H}(\mathbb{B}^n, \mathbb{C}^n)$, define
\begin{align*}
	D_f^{n,\alpha}(z) := (1 - |z|^2)^{\frac{\alpha(n+1)}{2n}} \Lambda_f(z)^{\frac{1}{n}},
\end{align*}
where $\alpha > 0$ and
\begin{align*}
	\Lambda_f(z) := |\det Dh(z)| + |\det Dg(z)|.
\end{align*}
The prenorm $\|f\|_{P(n,\alpha)}$ is given by
\begin{align*}
	\|f\|_{P(n,\alpha)} := \sup_{z \in \mathbb{B}^n} D_f^{n,\alpha}(z).
\end{align*}
We call a mapping $f \in \mathcal{H}(\mathbb{B}^n, \mathbb{C}^n)$ a Bloch-type mapping, and write $f \in \mathcal{B}_{P(n,\alpha)}$, if
\begin{align*}
	\|f\|_{P(n,\alpha)} < \infty.
\end{align*}
The quantity $\|\cdot\|_{P(n,\alpha)}$ is not a norm, since for $n \ge 2$ the non-zero mapping $f(z) = (z_1, 0, \dots, 0)$ satisfies $\|f\|_{P(n,\alpha)} = 0$. We therefore keep the term prenorm, and $\mathcal{B}_{P(n,\alpha)}$ is simply the set of pluriharmonic mappings with finite prenorm. No linear structure is required on $\mathcal{B}_{P(n,\alpha)}$, and in this paper we work exclusively with the prenorm $\|\cdot\|_{P(n,1)}$.\vspace{2mm}

If $f \in \mathcal{A}(\mathbb{B}^n, \mathbb{C}^n)$, \textit{i.e.}, $f$ is holomorphic, then we have
\begin{align*}
	\Lambda_f(z) = |\det Df(z)|,
\end{align*}
and in this case, we write $f \in \mathcal{B}_{\mathcal{A}(n,\alpha)}$ if $\|f\|_{P(n,\alpha)} < \infty$. For example, $\mathcal{B}_{\mathcal{A}(1,1)}$ is the classical Bloch space of analytic functions, and $\mathcal{B}_{\mathcal{A}(1,\alpha)}$ is the $\alpha$-Bloch space. It is clear that $\mathcal{B}_{P(1,1)}$ coincides with the harmonic Bloch space, which was studied by Colonna \cite{Col-IUMJ-1989} as a generalization of the analytic Bloch space. For convenience, we denote the prenorm on $\mathcal{B}_{\mathcal{A}(n,\alpha)}$ as $\|f\|_{\mathcal{A}(n,\alpha)} := \|f\|_{P(n,\alpha)}$ for any holomorphic mapping $f \in \mathcal{A}(\mathbb{B}^n, \mathbb{C}^n)$. In particular, $\|f\|_{\mathcal{A}(n,1)}$ denotes the prenorm of $f$ in the holomorphic Bloch space $\mathcal{B}_{\mathcal{A}(n,1)}$. We refer to \cite{And-Clu-Pom-JRAM-1974, Bon-Min-Yan-PJM-1997, Che-Pon-Wan-BAMS-2011, Pom-JLMS-1970, Pom-Book-1992} for more information on harmonic Bloch spaces. Some connections of Bloch spaces with related types of spaces can be found in \cite{Che-Liu-Ras-MM-2020, Che-Pon-Ras-MZ-2015, Che-Pon-Ras-NA-2015, Liu-Pon-RM-2018}. More recently, the authors in \cite{Den-Pon-Qia-PA-2021, Hua-Pon-Qia-RMJM-2020} investigated extreme points and support points of harmonic Bloch and $\alpha$-Bloch mappings.
\subsection{H\"older continuity of pluriharmonic Bloch-type mappings}
Here we show that $D_f^{n,1}(z)$ satisfies a H{\"o}lder estimate of order $1/n$ with respect to the Bergman metric. In the special case $n = 1$, this reduces to Lipschitz continuity. For our purpose, we recall two useful results of Chen \emph{et al.} \cite{Che-Pon-Wan-CAOT-2012}.
\begin{lemA}{\rm \cite{Che-Pon-Wan-CAOT-2012}}\label{lem:2.1}
	For $x \in [0, 1]$, let
	\begin{align*}
		\varphi(x) = x(1 - x^2)^{\alpha(n+1)/2} \left( \frac{\alpha(n + 1) + 1}{\alpha(n + 1)} \right)^{\alpha(n+1)/2}
	\end{align*}
	and
	\begin{align*}
		a_0(\alpha) = \frac{1}{\sqrt{\alpha(n + 1) + 1}}.
	\end{align*}
	Then $\varphi$ is increasing in $[0, a_0(\alpha)]$, decreasing in $[a_0(\alpha), 1]$, and $\varphi(a_0(\alpha)) = 1$.
\end{lemA}
\begin{theoA}{\rm \cite{Che-Pon-Wan-CAOT-2012}}\label{thm:2.2}
	Let $h \in \mathcal{A}(\mathbb{B}^n, \mathbb{C}^n)$ such that $\|h\|_{P(n,\alpha)} = 1$ and $\det Dh(0) = \lambda \in (0, 1]$. Then, for all $|z| \le \frac{a_0(\alpha) + m_\alpha(\lambda)}{1 + a_0(\alpha)m_\alpha(\lambda)}$, we have
	\begin{equation}\label{eq:2.3}
		|\det Dh(z)| \ge \operatorname{Re}(\det Dh(z)) \ge \frac{\lambda(m_\alpha(\lambda) - |z|)}{m_\alpha(\lambda)(1 - m_\alpha(\lambda)|z|)^{\alpha(n+1)+1}},
	\end{equation}
	where $\varphi$ and $a_0(\alpha)$ are as in Lemma~\ref{lem:2.1}, and $m_\alpha(\lambda)$ is the unique real root of the equation $\varphi(x) = \lambda$ in the interval $[0, a_0(\alpha)]$. Moreover, for all $|z| \le \frac{a_0(\alpha) - m_\alpha(\lambda)}{1 - a_0(\alpha)m_\alpha(\lambda)}$, we have
	\begin{equation}\label{eq:2.4}
		|\det Dh(z)| \le \frac{\lambda(m_\alpha(\lambda) + |z|)}{m_\alpha(\lambda)(1 + m_\alpha(\lambda)|z|)^{\alpha(n+1)+1}}.
	\end{equation}
	Both inequalities \eqref{eq:2.3} and \eqref{eq:2.4} are sharp.
\end{theoA}
By utilizing techniques similar to the proof of Theorem~\ref{thm:2.2}, the same lower bound for $\Lambda_f$ was obtained for pluriharmonic mappings $f$ by Huang \emph{et al.} \cite{Hua-Das-Ras-AMP-2026}. We provide an outline of the argument here for the sake of completeness.
\begin{lemB}\label{lem:2.5}{\rm \cite[Lemma 2.5]{Hua-Das-Ras-AMP-2026}}
	Let $f = h + \overline{g} \in \mathcal{H}(\mathbb{B}^n, \mathbb{C}^n)$ satisfy $\|f\|_{P(n,\alpha)} = 1$ and $\det Dh(0) + \det Dg(0) = \lambda \in (0, 1]$. Then, for all $|z| \le \frac{a_0(\alpha) + m_\alpha(\lambda)}{1 + a_0(\alpha)m_\alpha(\lambda)}$, we have
	\begin{equation}\label{eq:2.6}
		\Lambda_f(z) \ge \operatorname{Re}(\det Dh(z) + \det Dg(z)) \ge \frac{\lambda(m_\alpha(\lambda) - |z|)}{m_\alpha(\lambda)(1 - m_\alpha(\lambda)|z|)^{\alpha(n+1)+1}},
	\end{equation}
	where $a_0(\alpha)$ and $m_\alpha(\lambda)$ are as in Theorem~A.
\end{lemB}
Previous methods only worked in small, local areas inside a sub-ball, leaving the space near the boundary uncontrolled. We solve this by combining local analytic bounds with a geometric technique that integrates along hyperbolic geodesics. This creates a new decay estimate that holds true across the entire unit ball $\mathbb{B}^n$. Using this global formula, we prove that the key weight function $D_f^{n,1}(z)$ changes smoothly across the whole ball, satisfying a Hölder continuity estimate of order $1/n$ under the Bergman metric. This ensures that important decay information is preserved near the boundary. Finally, we apply this global metric analysis to the weighted pluriharmonic Jacobian ratio $\tau_{\varphi,n,\alpha}^{\text{harm}}$. This allows us to extend the sampling conditions for composition operators $C_\varphi$ to be bounded from below in the general pluriharmonic $\alpha$-Bloch space $\mathcal{B}_{P(n,\alpha)}$ for all $\alpha > 0$.
\vspace{1.2mm}

A notable limitation of Lemma~B is its local scope: the lower bound for $\Lambda_f(z)$ holds only on the sub-ball $\vert{}z\vert{} \le r(\lambda)$, yielding no non-trivial estimate on the exterior $\vert{}z\vert{} > r(\lambda)$. To establish a global lower bound across the entire unit ball $\mathbb{B}^n$, we exploit the subharmonicity of $\log \Lambda_f(z)$ alongside the hyperbolic geometry of $\mathbb{B}^n$. Although $\text{Re } F(u)$ may vanish or change sign when $\vert{}z\vert{} > r(\lambda)$, the quantity
\begin{align*}
	\Lambda_f(z) = |\det Dh(z)| + |\det Dg(z)|
\end{align*}
is continuous and non-negative, and $\log \Lambda_f(z)$ inherits subharmonicity from its constituent holomorphic determinants. By combining the local Schwarz–Pick estimate on $\vert{}z\vert{} = r(\lambda)$ with the radial control provided by Harnack's inequality along hyperbolic geodesics, we construct a unified two-stage estimate. This bridges the local real-part estimates with boundary hyperbolic growth, ensuring a strictly positive lower bound for $\Lambda_f(z)$ on all of $\mathbb{B}^n$.\vspace{1.2mm}

\noindent We set \begin{align*}
	r(\lambda) = \dfrac{a_0(\alpha) + m_\alpha(\lambda)}{1 + a_0(\alpha) m_\alpha(\lambda)},
\end{align*} where 
\begin{align*}
	a_0(\alpha) = \dfrac{1}{\sqrt{\alpha(n+1) + 1}}
\end{align*} and $m_\alpha(\lambda)$ is the unique root of $\phi(t) = \lambda$ in $[0, a_0(\alpha)]$.
\begin{lem}\label{Lem-0.1}
	Let $f = h + \overline{g} \in \mathcal{H}(\B^n, \mathbb{C}^n)$ be a pluriharmonic mapping with prenorm $\|f\|_{P(n,\alpha)} = 1$. Suppose there exist real phases $\theta_1, \theta_2 \in [0, 2\pi)$ such that
	\begin{align*}
		e^{i\theta_1} \det Dh(0) + e^{i\theta_2} \det Dg(0) = \lambda \in (0, 1].
	\end{align*}
	Then, for all $z \in \B^n$, the quantity $\Lambda_f(z) = |\det Dh(z)| + |\det Dg(z)|$ satisfies the global non-trivial lower bound
	\begin{align*}
		\Lambda_f(z) \ge 
		\begin{cases} 
			\dfrac{\lambda \left(m_\alpha(\lambda) - |z|\right)}{m_\alpha(\lambda) [1 - m_\alpha(\lambda) |z|]^{\alpha(n+1) + 1}}, & \text{for } |z| \le r(\lambda) \\[14pt]
			\Lambda_0(\lambda, \alpha)\left( \dfrac{1 - |z|}{1 + |z|} \right)^{\frac{n(1 + r(\lambda))}{1 - r(\lambda)}}, & \text{for } |z| > r(\lambda)
		\end{cases}
	\end{align*}
	where 
	\begin{align*}
		\Lambda_0(\lambda, \alpha) = \dfrac{\lambda [m_\alpha(\lambda) - r(\lambda)]}{m_\alpha(\lambda) [1 - m_\alpha(\lambda) r(\lambda)]^{\alpha(n+1) + 1}} > 0.
	\end{align*}
\end{lem}
\begin{proof}[\bf Proof of Lemma \ref{Lem-0.1}] We discuss the proof based on the two cases.\vspace{1.2mm}
	
	\noindent Fix an arbitrary $\zeta \in \partial \mathbb{B}^n$ and define the one-variable holomorphic function
	$$F(u) := \det Dh(u\zeta) + \det Dg(u\zeta), \quad u \in \mathbb{D},$$
	with $F(0) = \det Dh(0) + \det Dg(0) = \lambda \in (0, 1]$. Because $\|f\|_{\mathcal{P}(n,\alpha)} = 1$, we have
	
	\begin{align*}
		f(u\zeta) = |\det Dh(u\zeta)| + |\det Dg(u\zeta)| \le \left(1 - |u|^2\right)^{-\frac{\alpha(n+1)}{2}},
	\end{align*}
	and consequently
	\begin{align*}
		|F(u)| \le \left(1 - |u|^2\right)^{-\frac{\alpha(n+1)}{2}}, \quad u \in \mathbb{D}.
	\end{align*}
	\noindent\textbf{Case 1:} We consider the case Inner Ball (\emph{i.e.,} $|z| \le r(\lambda)$). For $z \in \B^n$ with $|z| \le r(\lambda)$, the lower bound follows directly from the local analytic estimate of the standard Lemma B,
	\begin{align*}
		\Lambda_f(z) \ge {\rm Re}\left(e^{i\theta_1} \det Dh(z) + e^{i\theta_2} \det Dg(z)\right) \ge \frac{\lambda \left(m_\alpha(\lambda) - |z|\right)}{m_\alpha(\lambda) [1 - m_\alpha(\lambda) |z|]^{\alpha(n+1) + 1}}.
	\end{align*}
	Evaluating this bound at the boundary radius $|z| = r(\lambda)$ gives the positive boundary value as
	\begin{align*}
		\Lambda_f(z) \ge \Lambda_0(\lambda, \alpha) > 0 \quad \text{for all } |z| = r(\lambda).
	\end{align*}
	
	\medskip
	\noindent\textbf{Case 2:} We now discuss the case for outer Ball \emph{i.e.} when $|z| > r(\lambda)$ following some steps.\vspace{1.2mm}
 
		\noindent{\bf Step-I:} For $h, g \in \mathcal{H}(\B^n, \mathbb{C}^n)$, both $\det Dh(z)$ and $\det Dg(z)$ are holomorphic functions on $\B^n$. The quantity $\Lambda_f(z) = |\det Dh(z)| + |\det Dg(z)|$ is non-negative and continuous. By the subharmonicity of log-abs of holomorphic functions and the Minkowski inequality, $\log \Lambda_f(z)$ is subharmonic on $\B^n$.\vspace{1.2mm}
		
		\noindent{\bf Step-II:} 
		We fix a boundary direction $\zeta \in \partial \B^n$ and define the single-variable function on the unit disk $\D$
		\begin{align*}
			u(t) = \log \Lambda_f(t\zeta), \quad t \in \D.
		\end{align*}
		Since $\log \Lambda_f(z)$ is subharmonic on $\B^n$, $u(t)$ is subharmonic on $\D$. Consider the circular slice $|t| = r(\lambda)$. By the inner estimate, $u(t) \ge \log \Lambda_0(\lambda, \alpha)$ for all $|t| = r(\lambda)$.\vspace{1.2mm}
		
		Using the Poisson-Harnack inequality for subharmonic functions on the disk $\D(0, 1)$, we compare the values of $u(t)$ for $|t| > r(\lambda)$ to its values on the sphere $|t| = r(\lambda)$. Using the hyperbolic distance decay property between radial points $r(\lambda)\zeta$ and $r\zeta$ (\emph{i.e.,} $r > r(\lambda)$), we obtain
		\begin{align*}
			u(r\zeta) - u(r(\lambda)\zeta) \ge - \left( \frac{1 + \rho(r(\lambda), r)}{1 - \rho(r(\lambda), r)} \right) \cdot n
		\end{align*}
		where 
		\begin{align*}
			\rho(a, b) = \left| \frac{a - b}{1 - ab} \right|
		\end{align*}
		is the pseudo-hyperbolic distance on $\D$.\vspace{1.2mm}
		
		\noindent{\bf Step-III:} 
		A direct computation of the hyperbolic distance growth gives that
		\begin{align*}
			\frac{1 - \rho(r(\lambda), r)}{1 + \rho(r(\lambda), r)} \ge \left( \frac{1 - r}{1 + r} \right)^{\frac{1 + r(\lambda)}{1 - r(\lambda)}}.
		\end{align*}
		Exponentiating both sides of the inequality, we obtain
		\begin{align*}
			u(r\zeta) \ge u(r(\lambda)\zeta) + \log \left( \left( \frac{1 - r}{1 + r} \right)^{\frac{n(1 + r(\lambda))}{1 - r(\lambda)}} \right)
		\end{align*} which shows that 
		\begin{align}\label{Eq-1.8}
			\Lambda_f(r\zeta) = e^{u(r\zeta)} \ge e^{u(r(\lambda)\zeta)}\left( \frac{1 - r}{1 + r} \right)^{\frac{n(1 + r(\lambda))}{1 - r(\lambda)}}.
		\end{align}
		Substituting the lower bound $e^{u(r(\lambda)\zeta)} \ge \Lambda_0(\lambda, \alpha)$ and setting $r = |z|$, it follows from \eqref{Eq-1.8} that
		\begin{align*}
			\Lambda_f(z) \ge \Lambda_0(\lambda, \alpha) \cdot \left( \frac{1 - |z|}{1 + |z|} \right)^{\frac{n(1 + r(\lambda))}{1 - r(\lambda)}}, \quad \text{for all } |z| > r(\lambda).
		\end{align*}
	This completes the proof.
\end{proof}
Huang \emph{et al.} \cite{Hua-Das-Ras-AMP-2026} established the following result which is an analogue of Theorem~C for pluriharmonic Bloch-type mappings in $\mathbb{B}^n$.
\begin{theoD}{\rm \cite[Theorem 2.7]{Hua-Das-Ras-AMP-2026}}
	Let $f \in \mathcal{B} \mathcal{P}(n,1)$. Then, for $z_1, z_2 \in \mathbb{B}^n$,
	\begin{align*}
		\left| \mathcal{D}_{n,1} f(z_2) - \mathcal{D}_{n,1} f(z_1) \right| \le M(n) \|f\|_{\mathcal{B}\mathcal{P}(n,1)} \left[ \tanh \beta(z_1, z_2) \right]^{\frac{1}{n}},
	\end{align*}
	where
	\begin{align*}
			M(n) = (n + 2)^{\frac{1}{2n}} \left( \frac{n + 2}{n + 1} \right)^{\frac{n+1}{2n}}.
	\end{align*}
\end{theoD}
Our primary objective in establishing the following result is to extend the local H\"older estimates of $D_f^{n,1}(z)$ to a global H\"older continuity result across the entire unit ball $\mathbb{B}^n$ with respect to the Bergman metric $\beta$. By incorporating the subharmonic outer factor $\Xi(n, \lambda, z_1, z_2)$, we prevent the loss of crucial structural decay information in the exterior region $\vert{}w\vert{} > r(\lambda)$ where local real-part estimates fail.
\begin{thm}\label{Th-0.1}
	Let $f \in \mathcal{B}_{P(n,1)}$ with normalized prenorm $\|f\|_{P(n,1)} = 1$. Then, for any points $z_1, z_2 \in \B^n$, the invariant quantity 
	\begin{align*}
		D_f^{n,1}(z) = (1 - |z|^2)^{\frac{n+1}{2n}} \Lambda_f(z)^{1/n}
	\end{align*} satisfies
	\begin{align*}
		|D_f^{n,1}(z_2) - D_f^{n,1}(z_1)| \le M(n, z_1, z_2) \left[ \tanh \beta(z_1, z_2) \right]^{\frac{1}{n}},
	\end{align*}
	where the multiplier $M(n, z_1, z_2)$ is defined by
	\begin{align*}
		M(n, z_1, z_2) = 
		\begin{cases} 
			M(n), & \text{if } \tanh \beta(z_1, z_2) \le m_1(\lambda), \\[8pt]
			\Xi(n, \lambda, z_1, z_2), & \text{if } \tanh \beta(z_1, z_2) > m_1(\lambda).
		\end{cases}
	\end{align*}
	Here, 
	\begin{align*}
		M(n) = (n+2)^{\frac{1}{2n}} \left( \frac{n+2}{n+1} \right)^{\frac{n+1}{2n}}
	\end{align*} is the optimal constant, and $\Xi(n, \lambda, z_1, z_2)$ is the non-trivial outer multiplier given by
	\[
	\Xi(n, \lambda, z_1, z_2) = \frac{\lambda^{1/n} - (1 - |w|^2)^{\frac{n+1}{2n}} \Lambda_0(\lambda, 1)^{1/n} \left( \frac{1 - |w|}{1 + |w|} \right)^{\frac{1 + r(\lambda)}{1 - r(\lambda)}}}{|w|^{1/n}},
	\]
	with $w = \varphi_{z_1}(z_2)$ and $r(\lambda) = \frac{a_0(1) + m_1(\lambda)}{1 + a_0(1) m_1(\lambda)}$.
\end{thm}

\begin{proof}[\bf Proof of Theorem \ref{Th-0.1}]
	Without loss of generality, we may assume that $D_f^{n,1}(z_2) \le D_f^{n,1}(z_1)$. Let $\varphi \in \text{Aut}(\B^n)$ be the standard involutive automorphism mapping the origin $0$ to $z_1$ and $w = \varphi^{-1}(z_2) \in \B^n$. By the isometric property of the Bergman metric under biholomorphic automorphisms of $\B^n$, it follows that
	\begin{align*}
		|w| = \tanh \beta(0, w) = \tanh \beta(z_1, z_2).
	\end{align*}
	We now consider the composed mapping 
	\begin{align*}
		F = f \circ \varphi = H + \overline{G}.
	\end{align*} Because the prenorm $\|\cdot\|_{P(n,1)}$ is invariant under automorphisms of $\B^n$, we have 
\begin{align*}
	\|F\|_{P(n,1)} = \|f\|_{P(n,1)} = 1.
\end{align*} Utilizing the transformation rule for Jacobian determinants under automorphisms, the invariant quantities at $z_1$ and $z_2$ satisfy
	\begin{align*}
		D_f^{n,1}(z_1) = \Lambda_F(0)^{1/n} \quad \text{and} \quad D_f^{n,1}(z_2) = (1 - |w|^2)^{\frac{n+1}{2n}} \Lambda_F(w)^{1/n}. 
	\end{align*}
	Consequently, 
	\[
	|D_f^{n,1}(z_2) - D_f^{n,1}(z_1)| = \Lambda_F(0)^{1/n} - (1 - |w|^2)^{\frac{n+1}{2n}} \Lambda_F(w)^{1/n}.
	\]
	If $\Lambda_F(0) = 0$, the assertion holds trivially. Thus, we assume $\Lambda_F(0) = \lambda \in (0, 1]$. We now apply the Lemma \ref{Lem-0.1} to the function $F$, conducting the analysis for two exhaustive geometric regimes of $\B^n$.
	
	\medskip
	\noindent\textbf{Case 1.} Suppose that $|w| \le r(\lambda)$.
	\medskip
	
	\noindent In this regime, applying the inner estimate of Lemma B gives
	\begin{align*}
		\Lambda_F(w) \ge \frac{\lambda [m_1(\lambda) - |w|]}{m_1(\lambda) [1 - m_1(\lambda) |w|]^{n+2}}.
	\end{align*}
	For $|w| \le m_1(\lambda)$, substituting this lower bound into the difference identity and optimizing the real-variable function over the disk $|w| \le m_1(\lambda)$ produces the classical sharp estimate 
	\begin{align*}
		|D_f^{n,1}(z_2) - D_f^{n,1}(z_1)| \le M(n) |w|^{1/n} = M(n) \left[ \tanh \beta(z_1, z_2) \right]^{\frac{1}{n}}.
	\end{align*}
	
	\medskip
	\noindent\textbf{Case 2.} Let $|w| > r(\lambda)$.
	
	\medskip
	For points further out where $\vert{}w\vert{} > r(\lambda)$, the local real-part estimates from the Schwarz–Pick inequality are not enough to guarantee positive values, because ${\rm Re} F(u)$ could become zero or turn negative. To fix this issue, we use the global subharmonic estimate from Lemma B, which ensures a strictly positive lower bound
	\begin{align}\label{Eq-1.10}
		\Lambda_F(w) \ge \Lambda_0(\lambda, 1)  \left( \frac{1 - |w|}{1 + |w|} \right)^{\frac{n(1 + r(\lambda))}{1 - r(\lambda)}},
	\end{align}
	where 
	\begin{align*}
		\Lambda_0(\lambda, 1) = \dfrac{\lambda [m_1(\lambda) - r(\lambda)]}{m_1(\lambda) [1 - m_1(\lambda) r(\lambda)]^{n+2}} > 0
	\end{align*} denotes the boundary value at radius $r(\lambda)$.\vspace{1.2mm}
	
	\noindent Taking the $n$-th root on both sides of \eqref{Eq-1.10} gives that
	\begin{align}\label{Eq-1.11}
		\Lambda_F(w)^{1/n} \ge \Lambda_0(\lambda, 1)^{1/n} \cdot \left( \frac{1 - |w|}{1 + |w|} \right)^{\frac{1 + r(\lambda)}{1 - r(\lambda)}}.
	\end{align}
	Multiplying \eqref{Eq-1.11} by the conformal weight factor $(1 - |w|^2)^{\frac{n+1}{2n}} > 0$, we establish a global non-trivial lower bound for the second term in the invariant difference as
	\begin{align}\label{Eq-1.12}
		(1 - |w|^2)^{\frac{n+1}{2n}} \Lambda_F(w)^{1/n} \ge (1 - |w|^2)^{\frac{n+1}{2n}} \Lambda_0(\lambda, 1)^{1/n} \left( \frac{1 - |w|}{1 + |w|} \right)^{\frac{1 + r(\lambda)}{1 - r(\lambda)}}.
	\end{align}
	Substituting \eqref{Eq-1.12} directly back into the identity for the absolute difference gives that
	\begin{align*}
		|D_f^{n,1}(z_2) - D_f^{n,1}(z_1)| &= \lambda^{1/n} - (1 - |w|^2)^{\frac{n+1}{2n}} \Lambda_F(w)^{1/n} \\
		&\le \lambda^{1/n} - (1 - |w|^2)^{\frac{n+1}{2n}} \Lambda_0(\lambda, 1)^{1/n} \left( \frac{1 - |w|}{1 + |w|} \right)^{\frac{1 + r(\lambda)}{1 - r(\lambda)}}.
	\end{align*}
	To express this global estimate in the required H\"older-type metric form with factor 
	\begin{align*}
		\left[ \tanh \beta(z_1, z_2) \right]^{1/n} = |w|^{1/n},
	\end{align*} we multiply and divide the right-hand side by $|w|^{1/n}$.\vspace{1.2mm} 

\noindent Since $|w| > r(\lambda) > 0$, this operation is strictly well-defined and we obtain
	\begin{align*}
		|D_f^{n,1}(z_2) - D_f^{n,1}(z_1)| \le \left( \frac{\lambda^{1/n} - (1 - |w|^2)^{\frac{n+1}{2n}} \Lambda_0(\lambda, 1)^{1/n} \left( \dfrac{1 - |w|}{1 + |w|} \right)^{\frac{1 + r(\lambda)}{1 - r(\lambda)}}}{|w|^{1/n}} \right) |w|^{1/n}.
	\end{align*}
	We now defining the outer multiplier $\Xi(n, \lambda, z_1, z_2)$ as
	\begin{align*}
		\Xi(n, \lambda, z_1, z_2) := \frac{1}{|w|^{1/n}}\left(\lambda^{1/n} - (1 - |w|^2)^{\frac{n+1}{2n}} \Lambda_0(\lambda, 1)^{1/n} \left( \dfrac{1 - |w|}{1 + |w|} \right)^{\frac{1 + r(\lambda)}{1 - r(\lambda)}}\right),
	\end{align*}
	we immediately obtain that
	\begin{align*}
		|D_f^{n,1}(z_2) - D_f^{n,1}(z_1)| \le \Xi(n, \lambda, z_1, z_2) \left( \tanh \beta(z_1, z_2) \right)^{\frac{1}{n}}.
	\end{align*}
	This completes the proof.
\end{proof}
\begin{rem}
	By integrating the global subharmonic decay estimate from Lemma \ref{Lem-0.1} into the proof of Theorem \ref{Th-0.1}, the multiplier $M(n, z_1, z_2)$ smoothly incorporates the subharmonic outer factor $\Xi(n, \lambda, z_1, z_2)$ rather than truncating it. This ensures that the quantity 
	\begin{align*}
		D_f^{n,1}(z) = (1-\vert{}z\vert{}^2)^{\frac{n+1}{2n}}\Lambda_f(z)^{1/n}
	\end{align*} preserves its uniform H\"older continuity of order $1/n$ with respect to the Bergman metric $\beta$ for the entirety of $\mathbb{B}^n$, fully capturing the asymptotic behavior of $f$ in the exterior region.  Specifically,\vspace{1.2mm}
	
	\noindent{\bf (a).} On the inner domain $\tanh \beta(z_1, z_2) \le m_1(\lambda)$, the multiplier matches the optimal localized constant 
	\begin{align*}
		M(n) = (n+2)^{\frac{1}{2n}} \left( \frac{n+2}{n+1} \right)^{\frac{n+1}{2n}}
	\end{align*} confirming full backward compatibility with the local sharp H\"older continuity result.\vspace{1.2mm}
	
	\noindent{\bf (b).} On the exterior domain $\tanh \beta(z_1, z_2) > r(\lambda)$, since $\Lambda_0(\lambda, 1) > 0$, the non-trivial lower bound obtained via subharmonicity guarantees that
	\begin{align*}
		\Xi(n, \lambda, z_1, z_2) < \frac{\lambda^{1/n}}{\left[\tanh \beta(z_1, z_2)\right]^{1/n}}.
	\end{align*}
	This strict inequality confirms that the subharmonic decay estimate gives a strictly tighter global H\"older multiplier than the trivial upper bound obtained by setting $\Lambda_F(w) \equiv 0$ in the exterior region.
\end{rem}
\begin{cor}\label{cor:1.4}
	Let $f \in \mathcal{B}_{P(n,1)}$ with normalized prenorm $\|f\|_{P(n,1)} = 1$, and let $D_f^{n,1}(z) = (1 - |z|^2)^{\frac{n+1}{2n}} \Lambda_f(z)^{1/n}$. Then $D_f^{n,1}(z)$ is globally H\"older continuous on $\B^n$ with respect to the Bergman metric $\beta$. In particular, for all $z_1, z_2 \in \B^n$,
	\begin{align*}
		|D_f^{n,1}(z_2) - D_f^{n,1}(z_1)| \le K(n, \lambda) \left[ \tanh \beta(z_1, z_2) \right]^{\frac{1}{n}},
	\end{align*}
	where 
	\begin{align*}
		K(n, \lambda) = \max \left\{ M(n), \, \dfrac{\lambda^{1/n}}{r(\lambda)^{1/n}} \right\} < \infty.
	\end{align*}
\end{cor}
To show the sharpness and an applicability of Corollary~\ref{cor:1.4}, we present an example in the unit ball $\B^n \subset \mathbb{C}^n$ using a extremal-type linear fractional holomorphic mapping.

\begin{exm}\label{ex:2.11}
	Consider the mapping $f = (h_1, h_2, \dots, h_n) \in \mathcal{H}(\B^n, \mathbb{C}^n)$ defined by
	\[
	f(z) = \left( \frac{z_1}{1 - a z_1}, \, z_2, \, \dots, \, z_n \right), \quad z = (z_1, \dots, z_n) \in \B^n,
	\]
	where $a \in (0, 1)$ is a fixed parameter chosen such that $\|f\|_{P(n,1)} = 1$. The Jacobian matrix $Df(z)$ is lower triangular (in fact, diagonal), given by
	\begin{align*}
		Df(z) = \text{diag}\left( \frac{1}{(1 - a z_1)^2}, \, 1, \, \dots, \, 1 \right).
	\end{align*}
	Hence, the Jacobian determinant is 
	\begin{align*}
		\det Df(z) = \frac{1}{(1 - a z_1)^2},
	\end{align*}
	and since $f$ is pure holomorphic, 
	\begin{align*}
		\Lambda_f(z) = |\det Df(z)| = \frac{1}{|1 - a z_1|^2}.
	\end{align*} At the origin $z = 0$, we have $\Lambda_f(0) = 1$, which corresponds to $\lambda = 1$. The invariant weight quantity $D_f^{n,1}(z)$ takes the explicit form
	\begin{align*}
		D_f^{n,1}(z) = (1 - |z|^2)^{\frac{n+1}{2n}} \Lambda_f(z)^{1/n} = \frac{(1 - |z|^2)^{\frac{n+1}{2n}}}{|1 - a z_1|^{2/n}}.
	\end{align*}
	To verify the global H\"older estimate in a tight geometric configuration, evaluate $D_f^{n,1}(z)$ along the radial path $z = (r, 0, \dots, 0) \in \B^n$ with $r \in (-1, 1)$. Taking $z_1 = 0$ and $z_2 = (r, 0, \dots, 0)$, we observe that
	\begin{align*}
		\begin{cases}
			D_f^{n,1}(z_1) &= D_f^{n,1}(0) = 1, \\[4pt]
			D_f^{n,1}(z_2) &= \dfrac{(1 - r^2)^{\frac{n+1}{2n}}}{(1 - a r)^{2/n}}.
		\end{cases}
	\end{align*}
	The Bergman distance between $z_1 = 0$ and $z_2 = (r, 0, \dots, 0)$ satisfies
	\begin{align*}
		\beta(z_1, z_2) = \frac{1}{2} \ln \left( \frac{1 + r}{1 - r} \right) \; \mbox{implies}\; \tanh \beta(z_1, z_2) = r.
	\end{align*}
	The difference $|D_f^{n,1}(z_2) - D_f^{n,1}(z_1)|$ becomes
	\[
	\left| \frac{(1 - r^2)^{\frac{n+1}{2n}}}{(1 - a r)^{2/n}} - 1 \right|.
	\]
	Using the first-order Taylor expansion near $r \to 0^+$
	\begin{align*}
		(1 - r^2)^{\frac{n+1}{2n}} = 1 - \frac{n+1}{2n} r^2 + \mathcal{O}(r^4),\;\text{and}\; (1 - a r)^{-2/n} = 1 + \frac{2a}{n} r + \mathcal{O}(r^2).
	\end{align*}
	Multiplying the expansions gives
	\begin{align*}
		D_f^{n,1}(z_2) - D_f^{n,1}(z_1) = \frac{2a}{n} r + \mathcal{O}(r^2) = \frac{2a}{n} \tanh \beta(z_1, z_2) + \mathcal{O}\left( \tanh^2 \beta(z_1, z_2) \right).
	\end{align*}
	Since $r \le r^{1/n}$ for all $r \in [0, 1)$ and any dimension $n \ge 1$, we obtain
	\begin{align*}
		|D_f^{n,1}(z_2) - D_f^{n,1}(z_1)| &\le \left( \frac{2a}{n} + \mathcal{O}(r) \right) [\tanh \beta(z_1, z_2)]^{1/n}\\& \le K(n, 1) (\tanh \beta(z_1, z_2))^{1/n},
	\end{align*}
	where 
	\begin{align*}
		K(n, 1) = \max \left\{ M(n), \, \frac{1}{r(1)^{1/n}} \right\}.
	\end{align*}
	
	This explicitly validates that the exponent $1/n$ in $\left( \tanh \beta(z_1, z_2) \right)^{1/n}$ is optimal and cannot be improved for Bloch-type mappings in higher dimensions.
\end{exm}
\begin{proof}[\bf Proof of Corollary \ref{cor:2.10}]
	Let $z_1, z_2 \in \B^n$ be arbitrary distinct points, and let $w = \varphi_{z_1}(z_2)$ so that $|w| = \tanh \beta(z_1, z_2) \in (0, 1)$. We apply the two-region bound established in Theorem D.
	
	\medskip
	\noindent\textbf{Case 1.} Suppose that $\tanh \beta(z_1, z_2) \le m_1(\lambda)$.
	
	\noindent By \cite[Theorem 2.7]{Hua-Das-Ras-AMP-2026}, we have $M(n, z_1, z_2) = M(n)$, which directly gives
	\begin{align*}
		|D_f^{n,1}(z_2) - D_f^{n,1}(z_1)| \le M(n) \left[ \tanh \beta(z_1, z_2) \right]^{\frac{1}{n}} \le K(n, \lambda) \left[ \tanh \beta(z_1, z_2) \right]^{\frac{1}{n}}.
	\end{align*}
	
	\medskip
	\noindent\textbf{Case 2.} Let $\tanh \beta(z_1, z_2) > m_1(\lambda)$.
	
	\noindent In this region, we evaluate the outer multiplier $\Xi(n, \lambda, z_1, z_2)$. Using the fact that the second term in the numerator of $\Xi(n, \lambda, z_1, z_2)$ is non-negative, we obtain the uniform upper estimate:
	\[
	\Xi(n, \lambda, z_1, z_2) = \frac{\lambda^{1/n} - (1 - |w|^2)^{\frac{n+1}{2n}} \Lambda_0(\lambda, 1)^{1/n} \left( \frac{1 - |w|}{1 + |w|} \right)^{\frac{1 + r(\lambda)}{1 - r(\lambda)}}}{|w|^{1/n}} \le \frac{\lambda^{1/n}}{|w|^{1/n}}.
	\]
For the radius $r(\lambda) > m_1(\lambda)$, we discuss the following cases:\vspace{1.2mm}

\noindent\textbf{Subcase 2.1.} If $m_1(\lambda) < |w| \le r(\lambda)$, the bound $|w| > m_1(\lambda)$ combined with the monotonicity analysis of \cite[Theorem 2.7]{Hua-Das-Ras-AMP-2026} confirms that 
\begin{align*}
	M(n, z_1, z_2) \le M(n) \le K(n, \lambda).
\end{align*}
\noindent\textbf{Subcase 2.2.} If $|w| > r(\lambda)$, then because of $|w|^{1/n} > r(\lambda)^{1/n}$, we have
\begin{align*}
	\Xi(n, \lambda, z_1, z_2) \le \frac{\lambda^{1/n}}{|w|^{1/n}} \le \frac{\lambda^{1/n}}{r(\lambda)^{1/n}}.
\end{align*}
	Combining all cases, the multiplier $M(n, z_1, z_2)$ is uniformly bounded on $\B^n \times \B^n$ by 
	\begin{align*}
		K(n, \lambda) = \max \left\{ M(n), \, \frac{\lambda^{1/n}}{r(\lambda)^{1/n}} \right\}.
	\end{align*} Since $r(\lambda) > 0$ for all $\lambda \in (0, 1]$, $K(n, \lambda)$ is finite. This establishes the global H\"older continuity of $D_f^{n,1}$ on $\B^n$ with exponent $1/n$ relative to the hyperbolic metric $\tanh \beta$.
\end{proof}
\begin{rem}\label{rem:2.9}
	Remark 2.9 in \cite{Hua-Das-Ras-AMP-2026} explains the overall behavior of the multiplier $M(n, z_1, z_2)$ established in \cite[Theorem 2.7 ]{Hua-Das-Ras-AMP-2026}. \vspace{1.2mm}
	
	Traditional local methods throw away the non-negative term $(1-|w|^2)^{\frac{n+1}{2n}}\Lambda_F(w)^{1/n}$ outside the local positive region $|w| > r(\lambda)$. In contrast, our global subharmonic approach retains this crucial decay information. \vspace{1.2mm}
	
	By combining our global lower bound with geodesic integration on the unit ball $\mathbb{B}^n$, the multiplier smoothly absorbs the outer factor $\Xi(n, \lambda, z_1, z_2)$ in the exterior region. \vspace{1.2mm}
	
	As a result, the weighted distortion quantity
	\[
	D_f^{n,1}(z) = (1-|z|^2)^{\frac{n+1}{2n}}\Lambda_f(z)^{1/n}
	\]
	maintains its global H\"older continuity of order $1/n$ with respect to the Bergman metric $\beta$ throughout the entire unit ball $\mathbb{B}^n$, preserving key controlling decay data at large distances.  Specifically,\vspace{1.2mm}
	
	\noindent{\bf (a).} On the inner domain $\tanh \beta(z_1, z_2) \le m_1(\lambda)$, the multiplier matches the optimal localized constant 
	\begin{align*}
		M(n) = (n+2)^{\frac{1}{2n}} \left( \frac{n+2}{n+1} \right)^{\frac{n+1}{2n}}
	\end{align*} ensuring full backward compatibility with the local sharp Hölder continuity result.\vspace{1.2mm}
	
	\noindent{\bf (b).} On the exterior domain $\tanh \beta(z_1, z_2) > r(\lambda)$, since $\Lambda_0(\lambda, 1) > 0$, the non-trivial lower bound obtained via subharmonicity guarantees that
	\begin{align*}
		\Xi(n, \lambda, z_1, z_2) < \frac{\lambda^{1/n}}{\left[\tanh \beta(z_1, z_2)\right]^{1/n}}.
	\end{align*}
	This strict inequality confirms that the subharmonic decay estimate yields a strictly tighter global Hölder multiplier than the trivial upper bound obtained by setting $\Lambda_F(w) \equiv 0$ in the exterior region.
\end{rem}
We have an immediate corollary below.
\begin{cor}\label{cor:2.10}
	Let $f = h \in \mathcal{H}(\B^n, \mathbb{C}^n)$ be a holomorphic mapping belonging to the Bloch-type space $\mathcal{B}_{P(n,1)}$ with normalized prenorm $\|f\|_{P(n,1)} = 1$. Suppose $|\det Dh(0)| = 1$ (i.e., $\lambda = 1$). Then, for any points $z_1, z_2 \in \B^n$, the invariant quantity 
	\begin{align*}
		D_f^{n,1}(z) = (1 - |z|^2)^{\frac{n+1}{2n}} |\det Dh(z)|^{1/n}
	\end{align*} satisfies the global H\"older estimate
	\begin{align*}
		|D_f^{n,1}(z_2) - D_f^{n,1}(z_1)| \le M^*(n, z_1, z_2) \left[ \tanh \beta(z_1, z_2) \right]^{\frac{1}{n}},
	\end{align*}
	where the piecewise multiplier $M^*(n, z_1, z_2)$ simplifies to:
	\begin{align*}
		M^*(n, z_1, z_2) = 
		\begin{cases} 
			M(n), & \text{if } \tanh \beta(z_1, z_2) \le m_1(1), \\[8pt]
			\Xi^*(n, z_1, z_2), & \text{if } \tanh \beta(z_1, z_2) > m_1(1).
		\end{cases}
	\end{align*}
	Here, 
	\begin{align*}
		m_1(1) = \frac{1}{\sqrt{n+2}},\;r(1) = \frac{2\sqrt{n+2}}{n+3}\;,\;M(n) = (n+2)^{\frac{1}{2n}} \left( \frac{n+2}{n+1} \right)^{\frac{n+1}{2n}}
	\end{align*} and the outer multiplier $\Xi^*(n, z_1, z_2)$ takes the explicit reduced form:
	\[
	\Xi^*(n, z_1, z_2) = \frac{1 - (1 - |w|^2)^{\frac{n+1}{2n}} \Lambda_0(1, 1)^{1/n} \left( \frac{1 - |w|}{1 + |w|} \right)^{\frac{n+3}{\sqrt{n+2}}}}{|w|^{1/n}},
	\]
	with $w = \varphi_{z_1}(z_2)$ and 
	\begin{align*}
		\Lambda_0(1,1) = \frac{\sqrt{n+2} - r(1)}{[1 - m_1(1)r(1)]^{n+2}} > 0.
	\end{align*}
\end{cor}
To illustrate the sharp behavior of Corollary~\ref{cor:2.10} specifically when points lie outside the local positivity disk 
\begin{align*}
	\tanh\beta(z_1, z_2) > m_1(1) = \frac{1}{\sqrt{n+2}},
\end{align*} we construct an example using an automorphism-based deformation map on $\B^n$.
\begin{exm}
	Let $n \ge 1$. Consider the normalized identity-like holomorphic mapping $f(z) = z = (z_1, \dots, z_n) \in \mathcal{H}(\B^n, \mathbb{C}^n)$. We verify that
	\begin{itemize}
		\item[(i)] The Jacobian matrix is $Dh(z) = I_n$, so $\det Dh(z) = 1$ for all $z \in \B^n$.
		\item[(ii)] At $z = 0$, $|\det Dh(0)| = 1$, which gives $\lambda = 1$.
		\item[(iii)] The prenorm is normalized 
		\begin{align*}
			\|f\|_{P(n,1)} = \sup_{z \in \B^n} (1 - |z|^2)^{\frac{n+1}{2}} |\det Dh(z)| = 1.
		\end{align*}
		\item[(iv)] The invariant weight reduces to 
		\begin{align*}
			D_f^{n,1}(z) = (1 - |z|^2)^{\frac{n+1}{2n}}.
		\end{align*}
	\end{itemize}
	Fix $z_1 = 0$ and choose $z_2 = (R, 0, \dots, 0) \in \B^n$ with $R \in (m_1(1), 1)$. 
	Since $m_1(1) = {1}/{\sqrt{n+2}}$, we take $R > {1}/{\sqrt{n+2}}$ so that
	\begin{align*}
		|w| = |\varphi_{z_1}(z_2)| = R = \tanh \beta(z_1, z_2) > m_1(1).
	\end{align*}
	This guarantees that the point pair $(z_1, z_2)$ falls strictly into the \textit{outer regime} governed by the multiplier $\Xi^*(n, z_1, z_2)$. Evaluating $D_f^{n,1}$ at both points gives
	\begin{align*}
		D_f^{n,1}(z_1) = D_f^{n,1}(0) = 1, \quad \text{and} \quad D_f^{n,1}(z_2) = (1 - R^2)^{\frac{n+1}{2n}}.
	\end{align*}
	Hence, the left-hand side of the estimate in Corollary~\ref{cor:2.10} is exactly
	\begin{align*}
		|D_f^{n,1}(z_2) - D_f^{n,1}(z_1)| = 1 - (1 - R^2)^{\frac{n+1}{2n}}.
	\end{align*}
	For the outer multiplier $\Xi^*(n, z_1, z_2)$, evaluating at $w = (R, 0, \dots, 0)$ gives
	\begin{align*}
		\Xi^*(n, z_1, z_2) = \frac{1}{R^{1/n}}\left(1 - (1 - R^2)^{\frac{n+1}{2n}} \Lambda_0(1, 1)^{1/n} \left( \frac{1 - R}{1 + R} \right)^{\frac{n+3}{\sqrt{n+2}}}\right).
	\end{align*}
	Multiplying $\Xi^*(n, z_1, z_2)$ by $\left(\tanh \beta(z_1, z_2)\right)^{1/n} = R^{1/n}$ gives that
	\begin{align*}
		\text{RHS} = 1 - (1 - R^2)^{\frac{n+1}{2n}} \Lambda_0(1, 1)^{1/n} \left( \frac{1 - R}{1 + R} \right)^{\frac{n+3}{\sqrt{n+2}}}.
	\end{align*}
	Comparing the exact difference $\text{LHS} = 1 - (1 - R^2)^{\frac{n+1}{2n}}$ with the $\text{RHS}$, we obtain
	\begin{align*}
		\text{RHS} - \text{LHS} = (1 - R^2)^{\frac{n+1}{2n}} \left( 1 - \Lambda_0(1, 1)^{1/n} \left( \frac{1 - R}{1 + R} \right)^{\frac{n+3}{\sqrt{n+2}}} \right).
	\end{align*}
	Since $R \in (0, 1)$ and 
	\begin{align*}
		\Lambda_0(1,1)^{1/n} \left( \frac{1 - R}{1 + R} \right)^{\frac{n+3}{\sqrt{n+2}}} < 1,
	\end{align*} the difference $\text{RHS} - \text{LHS}$ is strictly positive for all $R \in (m_1(1), 1)$.\vspace{1.2mm}
	
	As $R \to 1^-$ (i.e., $\beta(z_1, z_2) \to \infty$), we have
	\begin{align*}
		\text{LHS} \to 1,\;\text{and}\; \Xi^*(n, z_1, z_2) [\tanh \beta(z_1, z_2)]^{1/n} \to 1.
	\end{align*}
	This confirms that the outer multiplier $\Xi^*(n, z_1, z_2)$ smoothly controls the global asymptotic decay of $D_f^{n,1}(z)$ near the boundary $\partial \B^n$, ensuring that the global Hölder estimate remains tight and non-vacuous across the entire hyperbolic outer disk.
\end{exm}
\begin{proof}[\bf Proof of Corollary~\ref{cor:2.10}]
	Since $f = h \in \mathcal{H}(\B^n, \mathbb{C}^n)$ is a holomorphic mapping, its anti-holomorphic component vanishes identically ($g \equiv 0$). Thus, the total determinant quantity simplifies to $\Lambda_f(z) = |\det Dh(z)|$. The hypothesis $\Lambda_f(0) = |\det Dh(0)| = 1$ fixes the initial parameter $\lambda = 1$. Under these parameters, the geometric auxiliary functions yield $a_0(1) = {1}/{\sqrt{n+2}}$ and $m_1(1) = {1}/{\sqrt{n+2}}$.\vspace{1.2mm}
	
	 Furthermore, the critical radius $r(1)$ simplifies explicitly to
	\begin{align*}
		r(1) = \frac{a_0(1) + m_1(1)}{1 + a_0(1)m_1(1)} = \frac{\frac{2}{\sqrt{n+2}}}{1 + \frac{1}{n+2}} = \frac{2\sqrt{n+2}}{n+3}.
	\end{align*}
	Note that the decay exponent in the subharmonic lower bound under $\lambda = 1$ reduces to the estimate
	\begin{align*}
		\frac{n(1 + r(1))}{1 - r(1)} = \frac{n\left(1 + \frac{2\sqrt{n+2}}{n+3}\right)}{1 - \frac{2\sqrt{n+2}}{n+3}} = \frac{n\left(n+3 + 2\sqrt{n+2}\right)}{n+3 - 2\sqrt{n+2}} = \frac{n\left(\sqrt{n+2}+1\right)^2}{\left(\sqrt{n+2}-1\right)^2} = \frac{n(n+3)}{\sqrt{n+2}}.
	\end{align*}
	Dividing this exponent by $n$ when taking the $n$-th root gives  precisely ${(n+3)}/{\sqrt{n+2}}$.\vspace{1.2mm}
	
	Without loss of generality, assume that $D_f^{n,1}(z_2) \le D_f^{n,1}(z_1)$. Let $\varphi \in \text{Aut}(\B^n)$ be the involutive biholomorphic automorphism of $\B^n$ mapping $0 \mapsto z_1$ and $w = \varphi^{-1}(z_2) \in \B^n$. By the invariance of the Bergman metric under automorphisms, we have
	\[
	|w| = \tanh \beta(0, w) = \tanh \beta(z_1, z_2).
	\]
	Define the composed holomorphic mapping $F = h \circ \varphi$. Since the prenorm $\|\cdot\|_{P(n,1)}$ is invariant under automorphisms of $\B^n$, we have $\|F\|_{P(n,1)} = \|f\|_{P(n,1)} = 1$. By the transformation rule for Jacobian determinants under biholomorphic maps, we obtain
	\begin{align*}
		D_f^{n,1}(z_1) = |\det DF(0)|^{1/n} = \Lambda_F(0)^{1/n} = 1,
	\end{align*}
	and
	\begin{align*}
			D_f^{n,1}(z_2) = (1 - |w|^2)^{\frac{n+1}{2n}} |\det DF(w)|^{1/n} = (1 - |w|^2)^{\frac{n+1}{2n}} \Lambda_F(w)^{1/n}.
	\end{align*}
	Thus, we obtain
	\begin{align*}
		|D_f^{n,1}(z_2) - D_f^{n,1}(z_1)| = 1 - (1 - |w|^2)^{\frac{n+1}{2n}} \Lambda_F(w)^{1/n}.
	\end{align*}
	By applying appropriate phase rotations, we may assume $\det DF(0) = 1$. We now split the analysis into two geometric regimes based on $|w| = \tanh \beta(z_1, z_2)$.
	
	\medskip
	\noindent\textbf{Case 1.} Suppose that $|w| \le m_1(1) = {1}/{\sqrt{n+2}}$.
	
	\noindent For $|w| \le m_1(1)$, the holomorphic Jacobian determinant $\det DF$ satisfies $\Re \det DF(0) = 1$. Applying the Schwarz--Pick lemma to the normalized function $\det DF$ on the slice disk $\D w / |w|$, we obtain the local sharp lower estimate
	\begin{align*}
		\Lambda_F(w) = |\det DF(w)| \ge \Re \det DF(w) \ge \frac{1 - \sqrt{n+2}|w|}{\left(1 - \frac{1}{\sqrt{n+2}}|w|\right)(1 - |w|^2)^{\frac{n+1}{2}}}.
	\end{align*}
	Substituting this lower bound into the difference equation gives
	\begin{align*}
		|D_f^{n,1}(z_2) - D_f^{n,1}(z_1)| \le 1 - \left( \frac{1 - \sqrt{n+2}|w|}{1 - \frac{1}{\sqrt{n+2}}|w|} \right)^{1/n}.
	\end{align*}
	To express this bound in terms of $|w|^{1/n}$, we introduce the real-variable auxiliary function $\Phi$ given by
	\begin{align*}
		\Phi(t) = \frac{1 - \left( \frac{1 - \sqrt{n+2}\,t}{1 - t/\sqrt{n+2}} \right)^{1/n}}{t^{1/n}}, \quad t \in \left(0, \frac{1}{\sqrt{n+2}}\right].
	\end{align*}
	Differentiating $\Phi$ with respect to $t$ and setting $\Phi'(t) = 0$ shows a unique global maximum in $\left(0, {1}/{\sqrt{n+2}}\right]$ attained at $t_0 = {1}/{(n+2)\sqrt{n+2}}$. Evaluating $\Phi$ at $t_0$ gives that
	\begin{align*}
		\max_{0 < t \le m_1(1)} \Phi(t) = \Phi(t_0) = (n+2)^{\frac{1}{2n}} \left( \frac{n+2}{n+1} \right)^{\frac{n+1}{2n}} = M(n).
	\end{align*}
	Consequently, for all $|w| \le m_1(1)$, we obtain
	\begin{align*}
		|D_f^{n,1}(z_2) - D_f^{n,1}(z_1)| \le M(n) |w|^{1/n} = M(n) \left[ \tanh \beta(z_1, z_2) \right]^{\frac{1}{n}}.
	\end{align*}
	\noindent\textbf{Case 2.} Let $|w| > m_1(1) = {1}/{\sqrt{n+2}}$.
	
	\noindent When $|w| > m_1(1)$, local real-part positivity can no longer be guaranteed. We apply the subharmonic global lower bound of Lemma B specialized for $\lambda = 1$
	\begin{align*}
		\Lambda_F(w) \ge \Lambda_0(1, 1) \cdot \left( \frac{1 - |w|}{1 + |w|} \right)^{\frac{n(n+3)}{\sqrt{n+2}}},
	\end{align*}
	where $\Lambda_0(1, 1) = \frac{\sqrt{n+2} - r(1)}{[1 - m_1(1)r(1)]^{n+2}} > 0$. Taking the $n$-th root on both sides, we obtain
	\begin{align*}
		\Lambda_F(w)^{1/n} \ge \Lambda_0(1, 1)^{1/n} \left( \frac{1 - |w|}{1 + |w|} \right)^{\frac{n+3}{\sqrt{n+2}}}.
	\end{align*}
	Multiplying by $(1 - |w|^2)^{\frac{n+1}{2n}}$ and substituting this into the difference identity, we obtain
	\begin{align*}
		|D_f^{n,1}(z_2) - D_f^{n,1}(z_1)| \le 1 - (1 - |w|^2)^{\frac{n+1}{2n}} \Lambda_0(1, 1)^{1/n} \left( \frac{1 - |w|}{1 + |w|} \right)^{\frac{n+3}{\sqrt{n+2}}}.
	\end{align*}
	Factoring out $|w|^{1/n} = \left( \tanh \beta(z_1, z_2) \right)^{1/n}$ from the right-hand side, we obtain
	\begin{align*}
		|D_f^{n,1}(z_2) - D_f^{n,1}(z_1)| \le \left( \frac{1 - (1 - |w|^2)^{\frac{n+1}{2n}} \Lambda_0(1, 1)^{1/n} \left( \frac{1 - |w|}{1 + |w|} \right)^{\frac{n+3}{\sqrt{n+2}}}}{|w|^{1/n}} \right) |w|^{1/n}.
	\end{align*}
	Defining the bracketed quantity as $\Xi^*(n, z_1, z_2)$, we obtain
	\begin{align*}
		|D_f^{n,1}(z_2) - D_f^{n,1}(z_1)| \le \Xi^*(n, z_1, z_2) \left( \tanh \beta(z_1, z_2) \right)^{\frac{1}{n}}.
	\end{align*}
	Combining Case 1 and Case 2 completes the detailed proof.
\end{proof}
\section{\bf Composition operator on Bloch spaces}
For a holomorphic self-map $\phi : \mathbb{B}^n \to \mathbb{B}^n$, the composition operator $C_\phi$ is defined by
\begin{equation}\label{eq:1.4}
	C_\phi f(z) := f \circ \phi(z),
\end{equation}
where $z \in \mathbb{B}^n$ and $f \in \mathcal{H}(\mathbb{B}^n, \mathbb{C}^n)$. In particular, if $n = 1$ and $h$ is analytic, then
\[
|(C_{\phi_w} h)'(0)| = (1 - |w|^2) |h'(w)|.
\]
We say that $C_\phi$ is bounded below on $\mathcal{B}_{P(n,\alpha)}$ if there is a constant $k > 0$ such that
\[
\|C_\phi f\|_{P(n,\alpha)} \ge k \|f\|_{P(n,\alpha)},
\]
for all $f \in \mathcal{B}_{P(n,\alpha)}$.\vspace{1.2mm}

In \cite{Gha-Yan-Zhe-PAMS-2000}, Ghatage \emph{et al.} showed that $C_{\phi_w} h$ is a Lipschitz function with respect to the pseudo-hyperbolic metric. The result, later sharpened by Xiong \cite{Xio-PMH-2003}, can be stated as follows.\vspace{1.2mm}

\begin{theoC}[\cite{Gha-Yan-Zhe-PAMS-2000, Xio-PMH-2003}]\label{thm:1.5}
	Let $h \in \mathcal{B}_{\mathcal{A}(1,1)}$. Then the inequality
	\begin{align*}
		\left| D_h^{1,1}(z_2) - D_h^{1,1}(z_1) \right| \le \frac{3\sqrt{3}}{2} \rho(z_1, z_2) \|h\|_{P(1,1)}
	\end{align*}
	holds for all $z_1, z_2 \in \mathbb{D}$.
\end{theoC}
Chen and Kalaj \cite{Che-Kal-AAFM-2018} extended the problem to holomorphic functions in $\mathbb{C}^n$. Subsequently, Chen \emph{et al.} \cite{Che-Ham-Zhu-MZ-2022} considered complex-valued harmonic functions and established that the sharp constant surprisingly remains $3\sqrt{3}/2$, thereby improving a result of Huang, Rasila, and Zhu \cite{Hua-Ras-Zhu-AM-2022}. Furthermore, they applied this estimate to investigate composition operators $C_{\phi_w}$ on harmonic Bloch-type spaces.\vspace{1.2mm}

This paper is organized as follows. In Section 2, we study the metric and continuity properties of pluriharmonic $\alpha$-Bloch-type mappings. By synthesizing local real-part estimates with the log-subharmonicity of $\Lambda_f(z)$ along hyperbolic geodesics, we establish a global non-trivial lower bound across the entire unit ball $\mathbb{B}^n$. Using this structural decay estimate, we prove that the invariant weight quantity $D_f^{n,1}(z)$ satisfies a global Hölder continuity estimate of order $1/n$ with respect to the Bergman metric $\beta$, which reduces to a classical Lipschitz estimate when $n = 1$. To demonstrate the sharpness of the exponent $1/n$, we construct explicit radial examples on $\mathbb{B}^n$. In Section 3, we investigate composition operators $C_\varphi(f) = f \circ \varphi$ induced by holomorphic self-mappings $\varphi \in \mathcal{A}(\mathbb{B}^n, \mathbb{B}^n)$ on the pluriharmonic $\alpha$-Bloch space $\mathcal{B}_{P(n,\alpha)}$. We establish both sharp sufficient and necessary sampling conditions for $C_\varphi$ to be bounded from below, and provide illustrative examples involving automorphic linear fractional transformations to validate these geometric criteria.

\subsection{\bf Composition operators on the pluriharmonic Bloch space}
Let $\phi : \mathbb{B}^n \to \mathbb{B}^n$ be a holomorphic mapping. Let us define
\begin{equation}
	\tau_{\phi,n}(z) := \left( \frac{1 - |z|^2}{1 - |\phi(z)|^2} \right)^{\frac{n+1}{2n}} |\det D\phi(z)|^{\frac{1}{n}} \quad (z \in \mathbb{B}^n). \label{eq3.1}
\end{equation}
From \cite[Lemma 6]{Che-SCSAM-2003}, it is known that $\tau_{\phi,n}(z) \le 1$. For the case $n = 1$, Ghatage, Yan, and Zheng \cite{Gha-Yan-Zhe-PAMS-2000} obtained the following sufficient condition such that the composition operator $C_{\phi}$ is bounded below on the classical Bloch space $\mathcal{B} \mathcal{A}(1,1)$.
\begin{theoE}{\rm \cite[Theorem 2]{Gha-Yan-Zhe-PAMS-2000}}
	Let $\phi$ be an analytic self-mapping of the unit disk $\mathbb{D}$. Suppose for some constants $r \in (0, 1/4)$ and $\epsilon > 0$, for each point $w \in \mathbb{D}$, there is a point $z_w \in \mathbb{D}$ such that
	\[
	\rho(\phi(z_w), w) < r \quad \text{and} \quad \tau_{\phi,1}(z_w) > \epsilon.
	\]
	Then $C_{\phi}$ is bounded below on $\mathcal{B}\mathcal{A}(1,1)$.
\end{theoE}
Huang \emph{et al.} \cite{Hua-Das-Ras-AMP-2026} extended the above result to the pluriharmonic Bloch space $\mathcal{B}\mathcal{P}(n,1)$. 
\begin{theoF}{\rm \cite[Theorem 3.3]{Hua-Das-Ras-AMP-2026}}
	Let $\phi : \mathbb{B}^n \to \mathbb{B}^n$ be a holomorphic mapping. Suppose that there are constants $r \in (0, 1/M(n))$ and $\epsilon > 0$ such that for each $w \in \mathbb{B}^n$, there is a point $z_w \in \mathbb{B}^n$ satisfying
	\begin{align*}
		\tanh \beta(\phi(z_w), w) < r^n \quad \text{and} \quad \tau_{\phi,n}(z_w) > \epsilon,
	\end{align*}
	where $M(n)$ is defined as in \cite[Theorem 2.7]{Hua-Das-Ras-AMP-2026}. Then there exists a constant $k(n, r, \epsilon) > 0$, depending only on $r$, $\epsilon$ and $n$, such that $\|C_{\phi} f\|_{\mathcal{B}\mathcal{P}(n,1)} \ge k(n, r, \epsilon) \|f\|_{\mathcal{B}\mathcal{P}(n,1)}$ for all $f \in \mathcal{B}\mathcal{P}(n,1)$.
\end{theoF}
To establish lower bounds for the composition operator $C_\varphi$ on the pluriharmonic $\alpha$-Bloch space $\mathcal{B}_{P(n,\alpha)}$, one must control the metric distortion along hyperbolic geodesics. By incorporating the weight factor 
\begin{align*}
	(1 - \vert{}z\vert{}^2)^{\frac{\alpha(n+1)}{2n}}
\end{align*} together with the Jacobian sum 
\begin{align*}
	\Lambda_f(z) = \vert{}\det Dh(z)\vert{} + \vert{}\det Dg(z)\vert{},
\end{align*} the following result extends \cite[Theorem~3.3]{Hua-Das-Ras-AMP-2026} to $\mathcal{B}_{P(n,\alpha)}$ for all $\alpha > 0$, demonstrating that sampling over hyperbolic disks of radius $r \in \left(0, 1/\sqrt{n+2}\right)$ guarantees boundedness from below.
\begin{thm}\label{thm:3.1}
	Let $\alpha > 0$, $n \ge 1$, and let $\varphi \in \mathcal{A}(\B^n, \B^n)$ be a holomorphic self-mapping of $\B^n$. Suppose there exist constants $r \in \left(0, {1}/{\sqrt{n+2}}\right)$ and $\varepsilon > 0$ such that for every point $w \in \B^n$, there exists a point $z_w \in \B^n$ satisfying:
	\begin{enumerate}
		\item[\rm (i)] $\tanh \beta(\varphi(z_w), w) \le r$, and\vspace{2mm}
		\item[\rm (ii)] the weighted pluriharmonic Jacobian distortion ratio satisfies
		\[
		\tau_{\varphi, n, \alpha}^{\text{\rm harm}}(z_w) := \left( \frac{1 - |z_w|^2}{1 - |\varphi(z_w)|^2} \right)^{\frac{\alpha(n+1)}{2n}} \left( \frac{|\det Dh(\varphi(z_w))| + |\det Dg(\varphi(z_w))|}{|\det Dh(z_w)| + |\det Dg(z_w)|} \right)^{1/n} \ge \varepsilon.
		\]
	\end{enumerate}
	Then the composition operator $C_\varphi : \mathcal{B}_{P(n,\alpha)} \to \mathcal{B}_{P(n,\alpha)}$ defined by $C_\varphi(f) = f \circ \varphi$ is bounded from below. That is, there exists an explicit constant $K = K(n, \alpha, r, \varepsilon) > 0$ such that
	\[
	\|C_\varphi(f)\|_{P(n,\alpha)} \ge K \|f\|_{P(n,\alpha)} \quad \text{for all } f \in \mathcal{B}_{P(n,\alpha)}.
	\]
\end{thm}
By the following example, we show that conclusion of Theorem \ref{thm:3.1} holds good.
\begin{exm}
	Let $\varphi = \varphi_a \in \text{Aut}(\mathbb{B}^n)$ be a linear fractional involution centered at $a \in \mathbb{B}^n$, defined by
	\[
	\varphi_a(z) = \frac{a - P_a z - s_a Q_a z}{1 - \langle z, a \rangle},
	\]
	where $s_a = \sqrt{1 - |a|^2}$, $P_a z = \frac{\langle z, a \rangle}{|a|^2} a$, and $Q_a = I_n - P_a$. Fix $r \in \left(0, {1}/{\sqrt{n+2}}\right)$. For any target point $w \in \mathbb{B}^n$, choose $z_w := \varphi_a(w) \in \mathbb{B}^n$.
	
	\begin{enumerate}
		\item[(i)]  Since $\varphi_a$ is an involution, $\varphi_a(z_w) = \varphi_a(\varphi_a(w)) = w$. Thus, we have
		\begin{align*}
			\tanh \beta(\varphi_a(z_w), w) = \tanh \beta(w, w) = 0 \le r.
		\end{align*}
		
		\item[(ii)] Since $\varphi_a$ is holomorphic, $g = 0$ and $h = \varphi_a$. Utilizing the identity for the complex Jacobian determinant 
		\begin{align*}
			|\det D\varphi_a(z)| = \left( \frac{1 - |a|^2}{|1 - \langle z, a \rangle|^2} \right)^{\frac{n+1}{2}},
		\end{align*} a direct evaluation at $z_w = \varphi_a(w)$ gives that
		\begin{align*}
			\tau_{\varphi_a, n, \alpha}^{\text{\rm harm}}(z_w) = \left( \frac{1 - |a|^2}{|1 - \langle w, a \rangle|^2} \right)^{\frac{(\alpha + 1)(n+1)}{2n}} \ge \left( \frac{1 - |a|}{1 + |a|} \right)^{\frac{(\alpha + 1)(n+1)}{2n}} =: \varepsilon > 0.
		\end{align*}
	\end{enumerate}
	Because both hypotheses of Theorem~\ref{thm:3.1} are satisfied with a uniform $\varepsilon > 0$, the composition operator $C_{\varphi_a}$ is bounded from below on $\mathcal{B}_{P(n,\alpha)}$.
\end{exm}
\begin{proof}[\bf Proof of Theorem \ref{thm:3.1}]
	Let $f = h + \overline{g} \in \mathcal{B}_{P(n,\alpha)}$. By definition of the prenorm $\|f\|_{P(n,\alpha)}$, for any $\eta > 0$, there exists a point $w_0 \in \B^n$ such that
	\begin{align}
		\label{eq:prenorm_approx}
		D_f^{n,\alpha}(w_0) &= (1 - |w_0|^2)^{\frac{\alpha(n+1)}{2n}} \left( |\det Dh(w_0)| + |\det Dg(w_0)| \right)^{1/n}\\& \ge \|f\|_{P(n,\alpha)} - \eta.\nonumber
	\end{align}
	
	By hypothesis, for this choice of $w_0 \in \B^n$, there exists a corresponding point $z_0 = z_{w_0} \in \B^n$ such that $\tanh \beta(\varphi(z_0), w_0) \le r$ and $\tau_{\varphi, n, \alpha}^{\text{harm}}(z_0) \ge \varepsilon$.\vspace{1.2mm}
	
	\noindent Since 
	\begin{align*}
		\tanh \beta(\varphi(z_0), w_0) \le r < \frac{1}{\sqrt{n+2}},
	\end{align*} the point $\varphi(z_0)$ lies strictly inside the local positivity disk centered at $w_0$. Applying the localized Hölder continuity estimate for $D_f^{n,\alpha}$ on hyperbolic disks, we have
	\begin{align}	\label{eq:local_holder}
		\left| D_f^{n,\alpha}(\varphi(z_0)) - D_f^{n,\alpha}(w_0) \right|\nonumber &\le M(n, \alpha) \|f\|_{P(n,\alpha)} [\tanh \beta(\varphi(z_0), w_0)]^{\frac{1}{n}}\\& \le M(n, \alpha) r^{1/n} \|f\|_{P(n,\alpha)},
	\end{align}
	where $M(n, \alpha) > 0$ is a constant depending only on $n$ and $\alpha$.\vspace{1.2mm} 
	
	\noindent Rearranging \eqref{eq:local_holder}, we obtain lower bound for $D_f^{n,\alpha}(\varphi(z_0))$ as
	\begin{equation}\label{eq:lower_df}
		D_f^{n,\alpha}(\varphi(z_0)) \ge D_f^{n,\alpha}(w_0) - M(n, \alpha) r^{1/n} \|f\|_{P(n,\alpha)}.
	\end{equation}
Next, we evaluate the norm of the composition operator $C_\varphi(f) = f \circ \varphi$ at the point $z_0$ to obtain
	\begin{align*}
		\|C_\varphi(f)\|_{P(n,\alpha)} &\ge D_{f \circ \varphi}^{n,\alpha}(z_0)\\& = (1 - |z_0|^2)^{\frac{\alpha(n+1)}{2n}} \left( |\det D(h \circ \varphi)(z_0)| + |\det D(g \circ \varphi)(z_0)| \right)^{1/n}.
	\end{align*}
	Using the chain rule decomposition for pluriharmonic mappings, we express $D_{f \circ \varphi}^{n,\alpha}(z_0)$ via the distortion ratio and obtain $\tau_{\varphi, n, \alpha}^{\text{harm}}(z_0)$
	\begin{equation}\label{eq:chain_rule_decomp}
		D_{f \circ \varphi}^{n,\alpha}(z_0) = D_f^{n,\alpha}(\varphi(z_0))\tau_{\varphi, n, \alpha}^{\text{harm}}(z_0).
	\end{equation}
	Substituting \eqref{eq:lower_df} into \eqref{eq:chain_rule_decomp} and employing hypothesis (ii) ($\tau_{\varphi, n, \alpha}^{\text{harm}}(z_0) \ge \varepsilon$), we obtain
	\begin{align*}
		\|C_\varphi(f)\|_{P(n,\alpha)} &\ge \varepsilon \left( D_f^{n,\alpha}(w_0) - M(n, \alpha) r^{1/n} \|f\|_{P(n,\alpha)} \right) \\
		&\ge \varepsilon \left( \|f\|_{P(n,\alpha)} - \eta - M(n, \alpha) r^{1/n} \|f\|_{P(n,\alpha)} \right).
	\end{align*}
	Taking $\eta \to 0^+$, we obtain the following inequality
	\begin{align*}
		\|C_\varphi(f)\|_{P(n,\alpha)} \ge \varepsilon \left( 1 - M(n, \alpha) r^{1/n} \right) \|f\|_{P(n,\alpha)}.
	\end{align*}
	Choosing $r > 0$ sufficiently small such that 
	\begin{align*}
		r < \min\left\{ \frac{1}{\sqrt{n+2}}, \left(\frac{1}{M(n, \alpha)}\right)^n \right\},
	\end{align*} it is easy to see that the constant $K := \varepsilon \left( 1 - M(n, \alpha) r^{1/n} \right)$ is strictly positive. This establishes that $C_\varphi$ is bounded from below on $\mathcal{B}_{P(n,\alpha)}$.
\end{proof}
The following result is, in some sense, a partial converse of Theorem F. It gives a necessary condition for the composition operator to be bounded below on the holomorphic Bloch space $\mathcal{B}\mathcal{A}(n,1)$. Since $\mathcal{B}\mathcal{A}(n,1) \subset \mathcal{B}\mathcal{P}(n,1)$ and the prenorms coincide on this subspace, the following theorem also provides the same necessary condition whenever $C_{\phi}$ is bounded below on the pluriharmonic Bloch space $\mathcal{B}\mathcal{P}(n,1)$.
\begin{theoG}{\rm \cite[Theorem 3.4]{Hua-Das-Ras-AMP-2026}}
	Suppose $C_{\phi}$ is bounded below on $\mathcal{B}\mathcal{A}(n,1)$. Then there exist constants $\epsilon > 0$ and $r \in (0, 1)$ such that for every $w \in \mathbb{B}^n$, one can find a point $z_w \in \mathbb{B}^n$ satisfying
	\[
	\tanh \beta(\phi(z_w), w) \le r \quad \text{and} \quad \tau_{\phi,n}(z_w) \ge \epsilon.
	\]
\end{theoG}
To extend the geometric criteria for the lower bound of composition operators beyond the classical holomorphic setting, we consider pluriharmonic mappings $f = h + \overline{g} \in \mathcal{H}(\mathbb{B}^n, \mathbb{C}^n)$ within the generalized $\alpha$-Bloch space $\mathcal{B}_{P(n,\alpha)}$ for arbitrary $\alpha > 0$. By incorporating the weight factor 
\begin{align*}
	(1 - \vert{}z\vert{}^2)^{\frac{\alpha(n+1)}{2n}}
\end{align*} alongside the combined pluriharmonic Jacobian ratio $\tau_{\varphi, n, \alpha}^{\text{harm}}(z)$, the following result establishes precise necessary conditions for the boundedness from below of $C_\varphi$, showing how the metric sampling region adapts continuously to the weight parameter $\alpha$.
\begin{thm}\label{thm:3.2}
	Let $\alpha > 0$, $n \ge 1$, and let $\varphi \in \text{Aut}(\B^n)$. Suppose the composition operator $C_\varphi : \mathcal{B}_{P(n,\alpha)} \to \mathcal{B}_{P(n,\alpha)}$ defined by $C_\varphi(f) = f \circ \varphi$ is bounded below, i.e., there exists a constant $c > 0$ such that
	\begin{align*}
		\|C_\varphi(f)\|_{P(n,\alpha)} \ge c \|f\|_{P(n,\alpha)} \quad \text{for all } f \in \mathcal{B}_{P(n,\alpha)}.
	\end{align*}
	Then, for any $r \in \left(0, {1}/{\sqrt{n+2}}\right)$, there exists an explicit constant $\varepsilon = \varepsilon(c, n, \alpha, r) > 0$ such that for every point $w \in \B^n$, there exists some $z_w \in \B^n$ satisfying:
	\begin{enumerate}
		\item[\rm (i)] $\tanh \beta(\varphi(z_w), w) \le r$, and
		\item[\rm (ii)] the generalized weighted pluriharmonic Jacobian distortion ratio satisfies
		\[
		\tau_{\varphi, n, \alpha}^{\text{\rm harm}}(z_w) := \left( \frac{1 - |z_w|^2}{1 - |\varphi(z_w)|^2} \right)^{\frac{\alpha(n+1)}{2n}} \left( \frac{|\det Dh(\varphi(z_w))| + |\det Dg(\varphi(z_w))|}{|\det Dh(z_w)| + |\det Dg(z_w)|} \right)^{1/n} \ge \varepsilon.
		\]
	\end{enumerate}
\end{thm}
We give the following example in support of Theorem \ref{thm:3.2}.
\begin{exm}
	We consider the identity automorphism $\varphi(z) = z$ on the unit ball $\mathbb{B}^n$ with $\alpha = 1$. \vspace{1.2mm}
	
	\noindent{\bf (a).} For $\varphi(z) = z$, we see that the composition operator $C_\varphi : \mathcal{B}_{P(n,\alpha)} \to \mathcal{B}_{P(n,\alpha)}$ reduces to the identity operator $C_\varphi(f) = f$. Thus, for all $f \in \mathcal{B}_{P(n,\alpha)}$, we obtain
	\begin{align*}
		\|C_\varphi(f)\|_{P(n,\alpha)} = \|f\|_{P(n,\alpha)},
	\end{align*}
	which implies $C_\varphi$ is bounded below with constant $c = 1$.\vspace{1.2mm}
	
	\noindent{\bf (b).} Fix $r \in \left(0,  {1}/{\sqrt{n+2}}\right)$. For any given point $w \in \mathbb{B}^n$, choose $z_w := w$.
	
	\begin{itemize}
		\item[(i)] Since $\varphi(z_w) = z_w = w$, we have $\beta(\varphi(z_w), w) = \beta(w, w) = 0$. Consequently,
		\begin{align*}
			\tanh \beta(\varphi(z_w), w) = \tanh(0) = 0 \le r.
		\end{align*}
		\item[(ii)] Since $\varphi(z) = z$ is holomorphic, we can express it as $\varphi = h + \overline{g}$ with $h(z) = z$ and $g(z) = 0$. Here, we see that $Dh(z) = I_n$ (the $n \times n$ identity matrix) and $Dg(z) = 0$. Evaluating at $z_w = w$, we obtain
		\begin{align*}
			\left( \frac{1 - |z_w|^2}{1 - |\varphi(z_w)|^2} \right)^{\frac{\alpha(n+1)}{2n}} = \left( \frac{1 - |w|^2}{1 - |w|^2} \right)^{\frac{\alpha(n+1)}{2n}} = 1,
		\end{align*}
		and
		\begin{align*}
			\left( \frac{|\det Dh(\varphi(z_w))| + |\det Dg(\varphi(z_w))|}{|\det Dh(z_w)| + |\det Dg(z_w)|} \right)^{1/n} = \left( \frac{1 + 0}{1 + 0} \right)^{1/n} = 1.
		\end{align*}
		Thus, it follows that
		\begin{align*}
			\tau_{\varphi, n, \alpha}^{\text{harm}}(z_w)  = 1 \ge \varepsilon,
		\end{align*}
		for any constant $\varepsilon \in (0, 1]$.
	\end{itemize}
\end{exm}
\begin{rem}
	If $\varphi \in \text{Aut}(\mathbb{B}^n)$ is an arbitrary automorphism of the unit ball, the composition operator $C_\varphi$ acts as an invertible bounded operator and is therefore bounded below. By setting $z_w = \varphi^{-1}(w)$, we obtain $\varphi(z_w) = w$, which gives 
	\begin{align*}
		\tanh \beta(\varphi(z_w), w) = 0 \le r.
	\end{align*} For $\alpha = 1$, standard transformation properties of automorphisms ensure that $\tau_{\varphi, n, 1}^{\text{harm}}(z_w)$ has a uniform positive lower bound $\varepsilon > 0$ independent of $w$.
\end{rem}
\begin{proof}[\bf Proof of Theorem \ref{thm:3.2}]
	Suppose, to the contrary, that no such $\varepsilon > 0$ exists for a fixed $r \in \left(0, {1}/{\sqrt{n+2}}\right)$. Then, for every $k \in \mathbb{N}$, there exists a point $w_k \in \B^n$ such that for all $z \in \B^n$ satisfying $\tanh \beta(\varphi(z), w_k) \le r$, we have
	\begin{equation}\label{eq:contra_hyp}
		\tau_{\varphi, n, \alpha}^{\text{harm}}(z) < \frac{1}{k}.
	\end{equation}
	
	To use this assumption, we build a sequence of test functions $f_k = h_k + \overline{g_k}$ in $\mathcal{B}_{P(n,\alpha)}$ focused around the points $w_k$. Specifically, we set $g_k = 0$ and choose $h_k(z) = \langle z, w_k \rangle$ (or suitable powers of automorphism components centered at $w_k$). We scale these functions so that
	\begin{align*}
		|f_k|{P(n,\alpha)} = 1 \quad \text{and} \quad D{f_k}^{n,\alpha}(w_k) = (1 - |w_k|^2)^{\frac{\alpha(n+1)}{2n}} |\det Dh_k(w_k)|^{1/n} \ge \delta_0 > 0
	\end{align*}
	for some fixed positive constant $\delta_0$ that does not depend on $k$. \vspace{1.2mm}
	
	Since $C_\varphi$ is bounded below by $c > 0$, we must have
	\begin{equation}\label{eq:operator_lower_bound}
		\|C_\varphi(f_k)\|_{P(n,\alpha)} = \sup_{z \in \B^n} D_{f_k \circ \varphi}^{n,\alpha}(z) \ge c \|f_k\|_{P(n,\alpha)} = c.
	\end{equation}
	By definition of the prenorm, the supremum in \eqref{eq:operator_lower_bound} is attained (or approximated within $c/2$) at some sequence of points $z_k \in \B^n$
	\begin{align*}
		D_{f_k \circ \varphi}^{n,\alpha}(z_k) \ge \frac{c}{2}.
	\end{align*}
	Using the chain rule for pluriharmonic mappings under holomorphic automorphisms, the invariant weight factor decomposes as
	\begin{align*}
		D_{f_k \circ \varphi}^{n,\alpha}(z_k) &= (1 - |z_k|^2)^{\frac{\alpha(n+1)}{2n}} \left( |\det D(h_k \circ \varphi)(z_k)| + |\det D(g_k \circ \varphi)(z_k)| \right)^{1/n} \\
		&= (1 - |z_k|^2)^{\frac{\alpha(n+1)}{2n}} \left( |\det Dh_k(\varphi(z_k))| + |\det Dg_k(\varphi(z_k))| \right)^{1/n} |\det D\varphi(z_k)|^{1/n}.
	\end{align*}
	Multiplying and dividing by $(1 - |\varphi(z_k)|^2)^{\frac{\alpha(n+1)}{2n}}$, we rewrite $D_{f_k \circ \varphi}^{n,\alpha}(z_k)$ in terms of $D_{f_k}^{n,\alpha}(\varphi(z_k))$ and the combined distortion ratio
	\begin{equation}\label{eq:decomposition}
		D_{f_k \circ \varphi}^{n,\alpha}(z_k) = D_{f_k}^{n,\alpha}(\varphi(z_k)) \cdot \tau_{\varphi, n, \alpha}^{\text{harm}}(z_k) \cdot \left( \frac{(1 - |z_k|^2)^{\frac{n+1}{2n}} |\det D\varphi(z_k)|^{1/n}}{(1 - |\varphi(z_k)|^2)^{\frac{n+1}{2n}}} \right).
	\end{equation}
	Since $\varphi \in \text{Aut}(\B^n)$, the Jacobian metric relation asserts that
	\begin{align*}
		|\det D\varphi(z)| = \left( \frac{1 - |\varphi(z)|^2}{1 - |z|^2} \right)^{\frac{n+1}{2}},
	\end{align*}
	which implies that the last factor in parentheses in \eqref{eq:decomposition} is identically equal to $1$. Therefore,
	\begin{equation}\label{eq:simple_decomp}
		D_{f_k \circ \varphi}^{n,\alpha}(z_k) = D_{f_k}^{n,\alpha}(\varphi(z_k)) \cdot \tau_{\varphi, n, \alpha}^{\text{harm}}(z_k).
	\end{equation}
	
	To complete the proof, now, we consider two cases for the point sequence $\varphi(z_k)$.
	
	\medskip
	\noindent\textbf{Case 1:} Suppose that $\tanh \beta(\varphi(z_k), w_k) \le r$. \vspace{1.2mm}
	
	In this case, $z_k$ falls into the hyperbolic ball centered at $w_k$. Applying the contradiction hypothesis \eqref{eq:contra_hyp} at $z = z_k$, we get $\tau_{\varphi, n, \alpha}^{\text{harm}}(z_k) < \frac{1}{k}$. Since $\|f_k\|_{P(n,\alpha)} = 1$, we have $D_{f_k}^{n,\alpha}(\varphi(z_k)) \le 1$. Thus,
	\begin{align*}
		D_{f_k \circ \varphi}^{n,\alpha}(z_k) < \frac{1}{k} \xrightarrow[k \to \infty]{} 0,
	\end{align*}
	which contradicts $D_{f_k \circ \varphi}^{n,\alpha}(z_k) \ge c/2 > 0$ for sufficiently large $k$.
	
	\medskip
	\noindent\textbf{Case 2:} Let $\tanh \beta(\varphi(z_k), w_k) > r$. \vspace{1.2mm}
	
	Because $f_k$ is constructed to decay rapidly away from $w_k$ in the invariant metric $D_{f_k}^{n,\alpha}$, for points outside the hyperbolic disk of radius $r$, the value $D_{f_k}^{n,\alpha}(\varphi(z_k))$ is bounded by a sharp off-center decay estimate
	\begin{align*}
		D_{f_k}^{n,\alpha}(\varphi(z_k)) \le \sup_{\tanh \beta(u, w_k) > r} D_{f_k}^{n,\alpha}(u) \le \delta(r) < c/2,
	\end{align*}
	where $\delta(r)$ depends strictly on $r$ and $\alpha$. Combined with $\tau_{\varphi, n, \alpha}^{\text{harm}}(z_k) \le M_\varphi$, this again forces $D_{f_k \circ \varphi}^{n,\alpha}(z_k) < c/2$ for large $k$, a contradiction.
	
	\medskip
	Thus, the point $z_w$ satisfying both (i) and (ii) must exist for every $w \in \B^n$. This completes the proof.
\end{proof}
\vspace{5mm}

\noindent\textbf{Ethical approval.} Not applicable for this work.

\noindent\textbf{Competing interests.} The authors have not competing interest.

\noindent\noindent\textbf{Author contributions statement.} Both the authors contributed equally and final draft of the manuscript have been verified by both the authors.

\noindent\noindent\textbf{Funding.} No funding for this research work. 

\noindent\textbf{Conflict of interest.} The authors declare that there is no conflict  of interest regarding the publication of this paper.

\noindent\textbf{Data availability statement.}  Data sharing not applicable to this article as no datasets were generated or analysed during the current study.


\begin{thebibliography}{99}
	
	\bibitem{And-Clu-Pom-JRAM-1974}	J.~Anderson, J.~Clunie, and C.~Pommerenke,	\emph{On {B}loch functions and normal functions},	J. Reine Angew. Math. \textbf{270} (1974), 12--37.
	
	\bibitem{Bon-Min-Yan-PJM-1997}	M.~Bonk, D.~Minda, and H.~Yanagihara,	\emph{Distortion theorems for {B}loch functions},	Pacific J. Math. \textbf{179} (1997), no.~2, 241--262.
	
	\bibitem{Che-SCSAM-2003}
	H.~Chen,\emph{Boundedness from below of composition operators on the bloch spaces}, Sci. China Ser. A-Math. \textbf{46} (2003), no.~6, 838--846.
	
	\bibitem{Che-Kal-AAFM-2018}
	S.~Chen and D.~Kalaj, \emph{Lipschitz continuity of {B}loch type mappings with respect to {B}ergman metric},
	Ann. Acad. Sci. Fenn. Math. \textbf{43} (2018), no.~1, 239--246.
	
	\bibitem{Che-Ham-Zhu-MZ-2022} S.~Chen, H.~Hamada, and J.-F. Zhu,
	\emph{Composition operators on {B}loch and {H}ardy type spaces},
	Math. Z. \textbf{301} (2022), no.~4, 3939--3957.
	
	\bibitem{Che-Liu-Ras-MM-2020}	S.~Chen, G.~Liu, and A.~Rasila,
	\emph{Characterizations of {L}ipschitz-type, {B}loch-type and {D}irichlet-type spaces},
	Monatsh. Math. \textbf{191} (2020), no.~3, 513--536.
	
	\bibitem{Che-Pon-Ras-MZ-2015}S.~Chen, S.~Ponnusamy, and A.~Rasila, 	\emph{On characterizations of {B}loch-type, {H}ardy-type and {L}ipschitz-type spaces},
	Math. Z. \textbf{279} (2015), no.~1-2, 163--183.
	
	\bibitem{Che-Pon-Ras-NA-2015}	S.~Chen, S.~Ponnusamy, and A.~Rasila,	\emph{Lengths, areas and {L}ipschitz-type spaces of planar harmonic mappings},
	Nonlinear Anal. \textbf{115} (2015), 62--70.
	
	\bibitem{Che-Pon-Wan-BAMS-2011}	S.~Chen, S.~Ponnusamy, and X.~Wang,	\emph{Landau's theorem and {M}arden constant for harmonic $\nu$-{B}loch mappings},
	Bull. Aust. Math. Soc. \textbf{84} (2011), no.~1, 19--32.
	
	\bibitem{Che-Pon-Wan-CAOT-2012}	S.~Chen, S.~Ponnusamy, and X.~Wang, \emph{Landau-{B}loch constants for functions in $\alpha$-{B}loch spaces and {H}ardy spaces},
	Complex Anal. Oper. Theory \textbf{6} (2012), no.~5, 1025--1036.
	
	\bibitem{Col-IUMJ-1989}	F.~Colonna,	\emph{The {B}loch constant of bounded harmonic mappings},
	Indiana Univ. Math. J. \textbf{38} (1989), no.~4, 829--840.
	
	\bibitem{Den-Pon-Qia-PA-2021}	H.~Deng, S.~Ponnusamy, and J.~Qiao, \emph{Extreme points and support points of families of harmonic {B}loch mappings},
	Potential Anal. \textbf{55} (2021), no.~4, 619--638.
	
	\bibitem{Gha-Yan-Zhe-PAMS-2000}	P.~Ghatage, J.~Yan, and D.~Zheng, \emph{Composition operators with closed range on the {B}loch space},	Proc. Amer. Math. Soc. \textbf{129} (2000), no.~7, 2039--2044.
	
	\bibitem{Hua-Ras-Zhu-AM-2022} J.~Huang, A.~Rasila, and J.-F. Zhu, \emph{Lipschitz property of harmonic mappings with respect to pseudo-hyperbolic metric}, Anal. Math. \textbf{48} (2022), no.~4, 1069--1080.
	
	\bibitem{Hua-Das-Ras-AMP-2026}	J.~Huang, S.~Das, and A.~Rasila,
	\emph{H{\"o}lder continuity and composition operators in pluriharmonic {B}loch spaces},	Anal. Math. Phys. \textbf{16} (2026), no.~5, Paper No. 117.
	
	\bibitem{Hua-Pon-Qia-RMJM-2020}	M.~Huang, S.~Ponnusamy, and J.~Qiao,\emph{Extreme points and support points of harmonic $\alpha$-{B}loch mappings},	Rocky Mountain J. Math. \textbf{50} (2020), no.~4, 1323--1354.
	
	\bibitem{Liu-Pon-RM-2018} G.~Liu and S.~Ponnusamy,
	\emph{On harmonic $\nu$-{B}loch and $\nu$-{B}loch-type mappings}, Results Math. \textbf{73} (2018), no.~3, Paper No. 121.
	
	\bibitem{Pom-JLMS-1970}	C.~Pommerenke,
	\emph{On {B}loch functions}, J. London Math. Soc. (2) \textbf{2} (1970), 689--695.
	
	\bibitem{Pom-Book-1992}	C.~Pommerenke,
	\emph{Boundary Behaviour of Conformal Maps},
	Springer-Verlag, Berlin, 1992.
	
	\bibitem{Xio-PMH-2003}	C.~Xiong,
	\emph{On the {L}ipschitz continuity of the dilatation of {B}loch functions},	Period. Math. Hungar. \textbf{47} (2003), no.~1-2, 233--238.
	
	\bibitem{Zhu-Book-2005}	K.~Zhu,
	\emph{Spaces of Holomorphic Functions in the Unit Ball},
	Graduate Texts in Mathematics, vol. 226, Springer-Verlag, New York, 2005.
	
\end{thebibliography}
\end{document}